\documentclass[journal,10pt]{elsarticle}
\usepackage[OT1]{fontenc}  	
\usepackage[utf8]{inputenc}
\usepackage[greek,swedish,spanish,french,english,USenglish]{babel} 		
\usepackage{amssymb}  			
\usepackage{amsmath}	
\usepackage{amsthm}	
\usepackage{mathrsfs}	
\usepackage{esint}
\usepackage{units}
\usepackage{bbm}   				
\usepackage{color,graphicx} 			
\usepackage[small]{caption}
\usepackage{cases}			
\usepackage{mathpazo}  
\usepackage{bm}  
\usepackage{bigints_mod}	

\newcommand{\be}{\begin{equation}}
\newcommand{\ee}{\end{equation}}

\renewcommand{\Re}{\mathop{\rm Re}\nolimits}

\newcommand{\sh}{\mathop{\rm sh}\nolimits}
\newcommand{\ch}{\mathop{\rm ch}\nolimits}

\newcommand{\ctg}{\mathop{\rm ctg}\nolimits}

\newcommand{\pv}{\mathop{\rm p.v.}\nolimits}

\newcommand{\sgn}{\mathop{\rm sgn}\nolimits}

\newtheorem{theorem}{Theorem}
\newtheorem{lemma}{Lemma}
\newtheorem{corollary}{Corollary}

\newtheorem{remark}{Remark}

\makeatletter
\def\ps@pprintTitle{%
\let\@oddhead\@empty
\let\@evenhead\@empty
\let\@oddfoot\@empty
\let\@evenfoot\@oddfoot
}\makeatother

\begin{document}

\begin{frontmatter}

\title{On a generalization of Watson's trigonometric sum\\[1mm] 
 (on Dowker's sum of order one half)}

\author{Iaroslav V.~Blagouchine\corref{cor1}} 
\ead{iaroslav.blagouchine@univ-tln.fr}
\cortext[cor1]{Corresponding author.}

\begin{abstract}
In this paper we study the finite trigonometric sum $\sum a_l\csc\big(\pi l/n\big)\,$, 
where $a_l$ are equal to $\cos(2\pi l \nu/n)$ and
where the summation index $l$ and the discrete parameter $\nu$ both run through $1$ to $n-1$.
This sum is a generalization of \emph{Watson's trigonometric sum}, which has been 
extensively studied in a series of previous papers, and also may be regarded as
the so--called \emph{Dowker sum} of order one half.
It occurs in various problems in mathematics, physics and engineering,
and plays an important role in some number--theoretic problems.
In the paper, we obtain several integral and series representations for the above--mentioned sum, investigate its properties, 
derive various, including asymptotic, expansions for it, and deduce very accurate upper and lower bounds for it
(both bounds are asymptotically vanishing). 
In addition, we obtain two relatively simple approximate formulae containing only several terms, which are also very accurate
and can be particularly appreciated in applications.
Finally, we also derive several advanced summation formulae for the digamma functions, which relate 
the gamma and the digamma functions, the investigated sum, as well as the product of a sequence of cosecants
$\,\prod\big(\csc(\pi l/n)\big)^{\csc(\pi l/n)}$.
\end{abstract}

\begin{keyword}
finite trigonometric sums; cosecant sums; Watson's sum; Dowker's sum; Vinogradov sum, asymptotic representations;
asymptotic estimates; asymptotic expansions; digamma function; psi function, bounds; asymptotic estimates; approximations.
\end{keyword}

\end{frontmatter}

\section{Introduction}
\subsection{A short historical survey}
Finite trigonometric sums are an interesting object of study and often appear in analysis, discrete mathematics, 
combinatorics, number theory, applied statistics and in many other areas of mathematics.
They also often occur in applications, especially in physics and in a variety of related disciplines,
such as, for example, digital signal processing, computer science, information theory, telecommunications and cryptography,
see e.g.~\cite[Sect.~1.1]{iaroslav_18}.
Albeit many finite summation formulae are known and can be found in various handbooks and tables of series, 
see e.g.~\cite{gradstein_en,hansen_01,jolley_01,prudnikov_en}, such formulae still continue to attract the attention of mathematicians, 
see e.g.~\cite{allouche_04,allouche_05,annaby_01,beck_02,beck_01,berndt_04,berndt_05,berndt_06,bettin_01,bettin_02,byrne_01,chen_06,
chu_01,cvijovic_00,cvijovic_01,cvijovic_03,cvijovic_02,cvijovic_05,ejsmonta_01,fonseca_01,grabner_01,harshitha_01,he_01,raigorodskii_01,shkredov_01,stembridge_01}.
Indeed, very often, finite trigonometric sums cannot be evaluated in a closed--form at all.\footnote{By a closed--form expression
for the finite sum we mean a compact summation formula with a limited number of terms, which does not depend on the
length of the sum.} 
In such cases, it may be desirable to have a convenient asymptotic formula; notwithstanding, even the latter may be quite difficult.
We, for example, still do not know the asymptotics of many trigonometric sums related to the $\zeta$--function.

In 1916 the famous English mathematician George N.~Watson \cite{watson_02} considered the finite sum 
\be\label{984ycbn492}
S_n\equiv\sum_{l=1}^{n-1}  \csc\frac{\,\pi l\,}{n}\,,\qquad n\in\mathbbm{N}\setminus\{1\}\,,
\ee
which often occurs in mathematics, physics, and in a variety of related disciplines.
Watson have obtained the complete asymptotic expansion of this sum, allowing to calculate it quickly and 
accurately for large $n$ (see Remark~\ref{9843yndc3894} hereafter).
Watson, of course, was not the first who dealt with this sum; however,
for convenience, throughout the paper we refer to this sum as \emph{Watson's trigonometric sum}. 
The sum $S_n$ also appeared 
in many other works, including very recent ones, see e.g.~\cite[Sect.~1.1, p.~3]{iaroslav_18} for a nonexhaustive 
list of references. Furthermore, this sum has been found to be so remarkable, that Chen devoted a whole chapter 
of his book \cite[Chapt.~7]{chen_06} to it and to some of its properties.
Besides, as noted in \cite[Sec.~1]{tong_01}, there exist closed--form expressions for much more complicated sums
such as, for example, 
$\sum_{l=1}^{n-1} \csc^{2p}(\pi l/n)$, $p\in\mathbbm{N}$, or  $\sum_{l=1}^{n-1} \cos^{2q}(\theta l)\csc^{2p}(\pi l/n)$, 
$q,p\in\mathbbm{N}$,
see \cite[Chapt.~14]{chen_06}, \cite[Sec.~1]{tong_01}, \cite[vol.~1,\S~4.4.6]{prudnikov_en}, \cite{raigorodskii_01}, but still little is known about $S_n$.
In the above--cited paper \cite{iaroslav_18} we partially addressed this issue by examining in detail $S_n$ and by studying its properties, 
and also investigated its generalization
\be\label{894cbhc3ef}
\sum_{l=1}^{n-1}  \csc\!\left(\!\varphi+\frac{\,a\pi l\,}{n}\!\right) \,,\quad
\qquad \varphi+\frac{\,a\pi l\,}{n}\neq\pi k\,,\quad k\in\mathbbm{Z}\,, 
\ee
where $\varphi$ and $a$ are some parameters, the initial phase and the scaling factor respectively. The study of this sum
permitted to obtain, \emph{inter alia}, three advanced summation formulae for the digamma function.
It was also discovered that at large $n$ this sum may have qualitatively different behaviour, 
depending on how the initial phase and the scaling factor are chosen. 
It was established that \eqref{894cbhc3ef} has four qualitatively different leading terms, 
which, as $\varphi$ and $a$ start to vary,
appear or dissapear depending on the relationship between $\varphi$ and $a$; as a result, as $n$ increases, 
the sum $\sum\csc\big(\varphi+a\pi l/n\big)$ may become sporadically large \cite[Sect.~2.4.3]{iaroslav_18}. 
There also were other generalizations and extensions of $S_n$, which have been treated in the mathematical literature. 
For instance, in 1922 Hargreaves \cite{hargreaves_01},
extended Watson's investigations to the sum of cubes
\be\label{4o5iv4uyg}
\sum_{l=1}^{n-1}  \csc^3\!\frac{\,\pi l\,}{n}\,,\qquad n\in\mathbbm{N}\setminus\{1\}\,,
\ee
and obtained its first leading terms when $n$ becomes large
\be\label{4o5iv4uyg2}
\sum_{l=1}^{n-1}  \csc^3\!\frac{\,\pi l\,}{n}\,=\,\frac{2}{\,\pi^3\,}
\left\{n^3\zeta(3)+3n\zeta(2)\!\left(\!\ln n +\gamma-\ln\frac\pi2-\frac16\right) \!\right\}+\ldots
\ee
In 1924 Watson \cite{watson_02} considered more general sums 
\be\label{4o5iv}
S_n^{(r)}\equiv\sum_{l=1}^{n-1}  \csc^r\!\frac{\,\pi l\,}{n}\,,\quad\qquad n,r\in\mathbbm{N}\setminus\{1\}\,,
\ee
completed the result of Hargreaves by finding the complete asymptotics of $S_n^{(3)}$ at large $n$, proved that
\be\label{liuoy879}
\sum_{l=1}^{n-1}  \csc^r\!\frac{\,\pi l\,}{n}\,\sim\,\frac{\,2n^r\zeta(r)\,}{\pi^r}\,,
\quad\qquad n,r\in\mathbbm{N}\setminus\{1\}\,,
\quad n\to\infty\,,
\ee
and noted that for even $r$ the above sum should also have a closed--form (note also the the latter formula does not apply to the case $r=1$, 
that once again indicates that the sum $S_n^{(1)}=S_n$ represents quite a special case).\footnote{Note that there are two misprints on p.~580 \cite{watson_02}
in the unnumbered formulae defining $S_n^{(2r+1)}$ and $S_n^{(2r)}$: both sums should start with $m=1$ instead of $m=0.$}
Interestingly, the existence of
the closed--form expressions for such sums was also noted by Euler over 170 years before Watson; in particular,
the following formula 
\be\label{89ybwetv3}
\sum_{l=0}^{n-1} \csc^2\!\left(\!\varphi+\frac{\,\pi l\,}{n}\!\right) = \,n^2\csc^2 n\varphi\,,
\qquad n\in\mathbbm{N}\,,\quad \varphi\neq\pi(k- \nicefrac{l}{n})\,,
\ee
may be credited to Euler, see \cite[Sect.~1.1, pp.~5--6]{iaroslav_18}. Making $\varphi\to0$, we immediately get
\be\label{89ybwetv3b}
\sum_{l=1}^{n-1} \csc^2\!\frac{\,\pi l\,}{n}\, = \,\lim_{\varphi\to0}\!\left\{n^2\csc^2 n\varphi - \csc^2\varphi\right\}=
\,\frac{\,n^2-1\,}{3}\,,
\qquad n\in\mathbbm{N}\,,
\ee
which naturally agrees with Watson's asymptotics \eqref{liuoy879}.\footnote{As far as we know, formula \eqref{89ybwetv3b} 
has not been given by Euler explicitly in his works, albeit it is a particular case of his results, so that sometimes it 
may be attributed to other mathematicians, see e.g.~footnote 9 in \cite{iaroslav_18}, or Remark 2 in \cite{allouche_04} and the references therein.}
For higher even powers $r$ Euler did not explicitly gave closed--form expressions, but indicated a way for obtaining them.\footnote{The
reader interested in such a kind of expressions may find them, for example, in \cite[Chapt.~14]{chen_06}, 
\cite[vol.~1,\S~4.4.6]{prudnikov_en}, \cite{raigorodskii_01}, \cite{fonseca_01}.}
It may be also noted that Watson's results on $S_n^{(r)}$ were independently rediscovered several times, including rediscoveries of 
particular cases of \eqref{liuoy879}; for instance, 40 years later Gardner, Fisher and Carlitz obtained Watson's result \eqref{liuoy879}, but only 
for even $r$ \cite{gardner_03,fisher_04,chu_01,chu_02}.
Generalizations similar to \eqref{4o5iv} also appeared in works of Berndt, Alzer, Williams, Apostol, Chu, Marini and of some others, see 
e.g.~\cite{chu_01}, \cite[p.~267]{sofo_01}, \cite{fonseca_01}, and the references given therein. A wide class of sums generalizing $S_n$, 
\be\label{89y4bxc980yr}
\sum_{l=1}^{n-1} \cos\frac{2\pi \nu l}{n}\cdot\csc^{2q}\frac{\,\pi l\,}{n}\,,
\quad\qquad \nu=0,1,\ldots,n-1\,, \qquad q\in\mathbbm{N}\,,
\ee
was considered and used in some applications by Dowker \cite{dowker_00}, \cite[p.~772, Eq.~14]{dowker_01}, 
who, among other things, found a closed--form expression for it in terms of the Bernoulli polynomials of higher order
$B_n^{(s)}(x)$ \cite[Eqs.~4--7]{dowker_02}, and later, in terms of the generalized Bernoulli polynomials \cite[Eqs.~1--2]{dowker_03}, 
see Sect.~\ref{notation} hereafter. Sum \eqref{89y4bxc980yr} is also sometimes reffered to as \emph{Dowker's sum} 
\cite[Eq.~1.1]{cvijovic_01}, \cite[Eq.~1.5]{he_01}, \cite[Eq.~4.47]{fonseca_01}.
Somewhat similar generalizations were also considered by Chu and by some other authors, see e.g.~\cite{chu_01} and the 
references therein.
There exist, of course, many other generalizations of Watson's trigonometric sum, but it is worth noting that sums and series
with secants and cosecants are often very difficult to study, even asymptotically. 
For instance, Chu \cite[p.~137]{chu_01} studying sums similar to \eqref{89y4bxc980yr} remarks that 
\emph{``However, the resulting expressions will not be reproduced due to their complexity.''}
Furthermore, one can recall that we still do not know wherever the \emph{Flint Hills series}
$\,\sum l^{-3}\csc^2 l\,$ converges or not \cite[pp.~57--59 \& 265--268]{pickover_01},
\cite{alekseyev_01}, \cite{wolfram_03}. 

The aim of this paper is to investigate another generalization of Watson's sum \eqref{984ycbn492}, namely the trigonometric sum
\be\label{984ycbn492v3}
C_n(\nu)\,\equiv\sum_{l=1}^{n-1} \cos\frac{2\pi \nu l}{n}\cdot\csc\frac{\,\pi l\,}{n}\,,
\qquad\quad 
\nu=0,1,2,\ldots,n-1\,,
\ee
and its particular case
\be\label{984ycbn492v4}
C_n\,\equiv\sum_{l=1}^{n-1} \left(-1\right)^{l+1}\csc\frac{\,\pi l\,}{n}\,,
\ee
which may be obtained from $C_n(\nu)$ by setting $\nu=\frac12n$ when $n$ is even.
If $n$ is odd, then by symmetry $C_n=0.$ Note that the particular case $\nu=0$ is uninteresting for us,
because $C(n,0)=S_n$; the latter sum being already mentioned here and being extensively studied
in a series of previous papers. 
Thus, everywhere below, except if stated otherwise, we suppose that $\nu$ is not congruent to 0 modulo $n$.
Remark also that a similar sum with sines 
\be\label{984ycbn492v2}
\sum_{l=1}^{n-1} \sin\frac{2\pi \nu l}{n}\cdot\csc\frac{\,\pi l\,}{n}\,=\,0\,,
\qquad\quad 
\nu\in\mathbbm{Z}\,,
\ee
by virtue of symmetry. It is also worth noting that formally the sum $C_n(\nu)$ is Dowker's sum \eqref{89y4bxc980yr}
of order $q=\frac12$.
Besides, it may also be interesting to remark that formally, the sum $C_n(\nu)$ is also, 
up to one term and normalizing coefficients, the \emph{discrete cosine
transform} of a sequence of cosecants 
\be\notag
\left\{ \csc\frac{\,\pi\,}{n}\,,\;\csc\frac{\,2\pi l\,}{n}\,,\;\csc\frac{\,3\pi l\,}{n}\,,
\;\ldots\,,\;\csc\frac{\,\left(n-1\right)\pi l\,}{n}\right\},
\ee
and may, therefore, be regarded as the \emph{principal value} of this transform\footnote{The principal value
of a sum, $\pv$, is defined accordingly to \cite[Sect.~1.2]{iaroslav_18}.}.
Furthemore, since the similar sum of sines vanishes identically, $C_n(\nu)$ also gives, un to a coefficient,
the principal value of the \emph{discrete Fourier transform} and that of the \emph{discrete Hartley transform}.
The sums $C_n(\nu)$ and $C_n$ are not only interesting from the theoretical viewpoint, but also occur in applications.
For example, the former appears in an important number--theoretic problem related to the \emph{trigonometric P\'olya--Vinogradov sum} $f(n,k)$,
whose properties still remain little studied.\footnote{For further
reading on the trigonometric sum of P\'olya and Vinogradov, see e.g.~\cite[pp.~56 and 173--174]{vinogradov_01},
\cite[pp.~173--174]{apostol_01}, \cite{cochrane_01}, \cite{peral_01}, \cite{kongting_01}, \cite{cochrane_02}, \cite{alzer_01},
\cite{pomerance_01}.}
In fact, summing the geometric progression and then using the 
Fourier series expansion for $\left|\sin x\right|$, we see at once that 
\begin{eqnarray}
&\displaystyle
f(n,k)\,\equiv
\sum_{l=1}^{n-1} \left| \sum_{r=m}^{m+k-1} \exp\!\left(\frac{2\pi i l r}{n}\right) \right| \, = 
\sum_{l=1}^{n-1} 
\frac{\,\big|\sin(\pi l k/n)\big|\,}{\sin\left(\pi l/n\right)}  \notag\\[2.5mm]
&\displaystyle
= \,\frac{2 S_n}{\,\pi\,}\, 
-\, \frac{4}{\,\pi\,} \sum_{r=1}^{\infty} \frac{C_n(rk)}{\,4r^2-1\,}\,,\label{08undx3409}
\end{eqnarray}
where $n\in\mathbbm{N}\setminus\{1\}$, $m\in\mathbbm{N}$ and $k$ is a discrete parameter 
running through a complete residue system modulo $n$.

In the paper we provide several series and integral representations for $C_n(\nu)$ and $C_n$, establish their basic properties, 
obtain their asymptotic expansions (of two different kinds) and derive very accurate upper and lower bounds for them. 
We also obtain a useful approximate formula containing only three terms, which is very accurate and can be particularly appreciated in applications.
Much as in our previous research \cite{iaroslav_18}, we find that these sums are closely connected with the digamma function,
and with the square of the Bernoulli numbers, which is quite unusual.
Finally, we come to show that there exist summation relations between $C_n(\nu)$, the cosecant and its logarithm, as well as
the gamma and the digamma functions, \emph{viz.}
\begin{eqnarray}
&&\displaystyle \sum_{\nu=1}^{n-1}\Psi\!\left(\frac{\nu}{n}\right)\csc\frac{\,\pi \nu\,}{n}  \,= \,
- \left(\gamma+\ln2n\right)S_n  \, - \sum_{\nu=1}^{n-1}C_n(\nu)\ln\csc\frac{\pi \nu}{n}   \,, \label{i845nu8}  \\[3.5mm]
&&\displaystyle \sum_{\nu=1}^{n-1}\Psi\!\left(\frac{\nu}{n}\right)\csc\frac{\,\pi \nu\,}{n}\,=
\,-\left(\gamma+\ln2\pi n\right)S_n \, -\,2\!\sum_{\nu=1}^{n-1} \ln\Gamma\!\left(\frac{\nu}{n}\right) C_n(\nu)   \,, \label{iierwtgv}
\end{eqnarray}
\begin{eqnarray}
&&\displaystyle \sum_{\nu=1}^{n-1}\Psi\!\left(\frac{\nu}{n}\right)C_n(\nu)\,=
\,\left(\gamma+n\ln2\right)S_n\,\,-\,n \ln\prod_{\nu=1}^{n-1} \!\left(\csc\frac{\,\pi \nu\,}{n}\right)^{\!\csc\frac{\,\pi \nu\,}{n}}\,.\label{45g34g}
\end{eqnarray}
These relations complete the advanced summation formulae for the digamma function, which we obtained earlier in \cite[Appendix B, Eqs.~(B.6)--(B.11)]{iaroslav_07} and
in \cite[Eqs.~(12)--(15)]{iaroslav_18}. Two latter formulae are particularly beautiful. Equation \eqref{iierwtgv} relates
the values of the digamma function, evaluated at points uniformly distributed over the unit interval,
to those of the logarithm of the gamma function evaluated at same rational points; formula \eqref{45g34g} computes the product of a sequence 
of the form $a_l^{a_l}$, where each $a_l$ equals the cosecant of a rational part of $\pi$.

\subsection{Notation and conventions}\label{notation}
By definition, the set of natural numbers $\mathbbm{N}$ does not include zero. 
The symbol $\,\equiv\,$ means ``is defined by definition as'' and should be not confused 
with the conguence symbol from modular arithmetic.
Different special numbers are denoted as follows:
$\,\gamma\equiv\lim_{n\to\infty}\!\big(H_n-\ln n\big)=0.5772156649\ldots\,$
is the Euler constant, $\,H_n\equiv1+\nicefrac{1}{2}+\ldots+\nicefrac{1}{n}\,$
stands for the $n$th harmonic number,
${B}_n$ denotes the $n$th Bernoulli number. In particular
${B}_0=+1$, ${B}_1=\nicefrac{-1}{2}$, ${B}_2=\nicefrac{+1}{6}$,
${B}_3=0$, ${B}_4=\nicefrac{-1}{30}$, ${B}_5=0$, ${B}_6=\nicefrac{+1}{42}$, ${B}_7=0$,
${B}_8=\nicefrac{-1}{30}$, ${B}_9=0$, ${B}_{10}=\nicefrac{+5}{66}$,
${B}_{11}=0$, ${B}_{12}=\nicefrac{-691}{2730}\,,\ldots$\footnote{For 
further values and definitions, see \cite[Chapt.~2, \S~1]{norlund_02}, \cite[Chapt.~1, \S~1.1]{krylov_01}, \cite[Sect.~23]{abramowitz_01}, 
Note also that there exist slightly different definitions for the Bernoulli numbers, see e.g.~\cite[p.~91]{hagen_01},
\cite[pp.~32, 71]{lindelof_01}, \cite{watson_02,williams_01}, or \cite[pp.~3--6]{arakawa_01}.}
The Bernoulli numbers are the particular values of the Bernoulli polynomials, which are denoted by $B_n(x)$; these polynomials are defined either implicitly
via their generating function 
\be\notag
\frac{\,z\,e^{xz}}{\,e^z-1\,}\,=\sum_{n=0}^\infty \frac{B_n(x)}{n!} \, z^n\,,\quad\qquad |z|<2\pi\,,
\ee
or explicitly via the Bernoulli numbers and the binomial coefficients $\binom{n}{k}$
\be\notag
B_n(x)\,=\,B_n \,+\sum_{k=0}^{n-1} \binom{n}{k} \, x^{n-k} B_k\,.
\ee
Both definitions imply that $B_n=B_n(0)$. Furthermore, $B_n(x)$ are themselves the particular
case of the Bernoulli polynomials of higher order $B_n^{(s)}(x)$, \emph{viz.} $B_n(x)=B_n^{(1)}(x)$, see \cite[Chapt.~1, \S~1.2]{krylov_01},
\cite[Chapt.~2 \& 6]{norlund_02}, \cite[pp.~127--135]{milne_01}, \cite[p.~323, Eq.~(1.4)]{carlitz_01}, \cite{norlund_03}, 
\cite{norlund_01}, \cite[Sect.~23]{abramowitz_01}, \cite[Vol.~III, \S~19.7]{bateman_01}, \cite[Eq.~(4)]{cvijovic_04}, \cite[p.~16, Eq.~(52)]{iaroslav_10}.
We also make use of numerous abbreviations for the functions and series.
In particular, $\lfloor z\rfloor$ denotes the integer part of $z$, $\delta_{k,l}$ is the Kronecker delta of discrete variables $k$ and $l$,
$\operatorname{tg}z$, $\operatorname{ctg}z$, $\operatorname{ch}z$ and $\operatorname{sh}z$ stand for the tangent of $z$,
the cotangent of $z$, the hyperbolic cosine of $z$ and the hyperbolic sine of $z$ respectively.
Writings $\Gamma(s)$, $\Psi(s)$, $\Psi_n(s)$ and $\zeta(s)$ denote the gamma ($\Gamma$) function, the digamma (or $\Psi$) function,
the polygamma function of order $n$ and the Euler--Riemann zeta ($\zeta$) function of argument $s$ respectively.
In order to be consistent 
with the previous notation of Watson \cite{watson_02,watson_03} and of some other authors,
we denote by $S_n$ \emph{Watson's trigonometric sum} \eqref{984ycbn492}, and
by $C_n(\nu)$ and $C_n$ the sums \eqref{984ycbn492v3} and \eqref{984ycbn492v4} respectively. Occasionally, 
we make use of \emph{divergent series}; summability, summations methods and their regularity 
are defined accordingly to Hardy's monograph \cite{hardy_02}.
It is also supposed that the reader possesses, up to a certain point, a working knowledge of the theory of 
asymptotic expansions \cite{copson_01,dingle_01,erdelyi_01,evgrafov_03_eng,evgrafov_01_eng,olver_01}; 
the order symbols $O$, $o$ and the asymptotic equivalence symbol $\sim$ are defined accordingly to Evgrafov's and Erd\'elyi's books
\cite[Chapt.~1, \S~4]{evgrafov_01_eng}, \cite[Chapt.~1]{erdelyi_01}. 
By the error between the quantity $A$
and its approximated value $B$, we mean $A-B$. 
Finally, all the Figures in the paper were exported from the CAS Maple. The data used to trace the graps
were calculated with the $50$-digit ``precision'', enabled with the help of the command \texttt{Digits:=50}; some graphs and results were also 
verified independently with the help of \textsc{Matlab} and the CAS \mbox{Mathematica.}\footnote{It should be noted that such a ``precision''
($50$ digits) is not guaranteed at all by Maple. For instance, by computing the approximation error for the graph in Fig.~\ref{937fybfe5} with 
the $8$--digit precision (which should, in principle, be largely sufficient to corectly trace a graph), we first obtained completely 
erroneous values. Writing in Maple 12 \texttt{evalf($\epsilon$(300,34),8)}, where $\epsilon(n,\nu)$ is the approximation error, 
returned the value \texttt{-0.3136e-4}, while the correct value is close to \texttt{-5e-011}. In other words, in some situations, 
Maple 12 fails to correctly compute even the order of magnitude when using the $8$-digit precision.}
\looseness=-1

\section{Basic properties}\label{h09387rxhxdws}
First of all, we note that albeit $\nu$ may take only discrete values,
if $\nu$ was a continuous complex variable, then $C_n(\nu)$ would be analytic
in the whole complex plane, except the case $n\to\infty$.
Furthermore, $C_n(\nu)$, defined in \eqref{984ycbn492v3} only for $\,\nu=0,1,2,\ldots,n-1\,$, 
may easily be extended to any integer $\nu$ by means of the formulae
\be\label{9348yb394y}
C_n(\nu+mn)\,=\,
\left\{
\begin{array}{llll}
\, S_n\,,\qquad 		& \nu=0 \,,\qquad 			& m\in\mathbbm{Z}\,,\\[4mm]
\,C_n(\nu)\,,\qquad 		& \nu=1,2,\ldots,n-1\,,\qquad 	& m\in\mathbbm{Z}\,,
\end{array}
\right.
\ee
In addition, $C_n(\nu)$ has also the following basic properties:
\begin{eqnarray}
&&\displaystyle C_n(\nu)\,=\, C_n(n-\nu) \,, \qquad\quad\qquad  C_{2n-1} =0\,,\label{39ufif}\\[3.5mm]
&&\displaystyle -C_{2n}\leqslant C_{2n}(\nu)\,, \qquad\qquad C_n(\nu)\leqslant S_n\,, 
\qquad\qquad 0< C_{2n}< S_{2n}\,,   \label{98743cb38} \\[3.5mm]
&&\displaystyle C_n(\nu+1)\,=\, C_n(\nu) \,-\, 2\ctg\frac{\, \left(2\nu+1\right)\pi\,}{2n} \,, \label{3984ycbf34y}\\[3.5mm]
&&\displaystyle C_n(\nu+\kappa)\,=\, C_n(\nu) \,-\, 2\sum_{l=1}^\kappa\ctg\frac{\, \left(2\nu+2l-1\right)\pi\,}{2n} \,,
\qquad\quad\kappa\in\mathbbm{N}\,, \label{jkhbwc8qfc032}\\[3.5mm]
&&\displaystyle C_n(\nu+\kappa)\,=\, C_n(\kappa) \,-\, 2\sum_{l=1}^\nu\ctg\frac{\, \left(2\kappa+2l-1\right)\pi\,}{2n} \,,
\qquad\quad\kappa\in\mathbbm{N}\,, \label{834cn384d3f}\\[3.5mm]
&&\displaystyle C_n(2\nu)\,=\, C_n(\nu) \,-\, 2\sum_{l=1}^\nu\ctg\frac{\, \left(2\nu+2l-1\right)\pi\,}{2n} \,, \label{iwehnc83}\\[3.5mm]
&&\displaystyle C_n(\nu)-C_n(\kappa)\,=\, -\, 2\!\!\!\sum_{l=\kappa+1}^\nu \!\! \ctg\frac{\, \left(2l-1\right)\pi\,}{2n} \,, 
\qquad\quad \nu>\kappa \,, \label{iercrwev}\\[3.5mm]
&&\displaystyle \sum_{\nu=1}^{n-1}C_n(\nu)\,=\, -S_n \,, \qquad\qquad\quad
\sum_{\nu=1}^{n-1}C^2_n(\nu)\,=\, \frac{\,n\left(n^2-1\right)\,}{3}  - \,S^2_n \,,\label{4938cn349y}\\[3.5mm]
&&\displaystyle \sum_{\nu=1}^{n-1}C_n(\nu)\cos\frac{2\pi \nu k}{n}\,=\, n\csc\frac{\,\pi k\,}{n} -S_n \,, \qquad
\sum_{\nu=1}^{n-1}C_n(\nu)\sin\frac{2\pi \nu k}{n}\,=\, 0 \,,\label{4ouedy2}\\[3.5mm]
&&\displaystyle \sum_{\nu=1}^{n-1}C_n(\nu)\ctg\frac{\pi \nu}{n}\,=\, 0 \,, \qquad\qquad\qquad
\sum_{\nu=1}^{n-1} \nu \,C_n(\nu)\,=\, -\frac{\,n \,S_n\,}{2} \,,\label{4ouedy}\\[3.5mm]
&&\displaystyle \sum_{\nu=1}^{n-1}C_n(\nu)\ln\sin\frac{\pi \nu}{n}\,=\, \left(\gamma+\ln2n\right)S_n\,
+\sum_{r=1}^{n-1}\Psi\!\left(\frac{r}{n}\right)\csc\frac{\,\pi r\,}{n} \,, \label{iwytb877}  \\[3.5mm]
&&\displaystyle \sum_{\nu=1}^{n-1}C_n(\nu)\ln\Gamma\!\left(\frac{\nu}{n}\right)\,=
\,-\frac{\,\left(\gamma+\ln2\pi n\right)S_n\,}{2}\,-\,\frac{1}{\,2\,}\!\sum_{\nu=1}^{n-1}
\Psi\!\left(\frac{\nu}{n}\right)\csc\frac{\,\pi \nu\,}{n}   \,, \label{iiytb8t3} \\[3.5mm]
&&\displaystyle \sum_{\nu=1}^{n-1}C_n(\nu)\Psi\!\left(\frac{\nu}{n}\right)\,=
\,\left(\gamma+n\ln2\right)S_n\,\,-\,n\sum_{\nu=1}^{n-1} \csc\frac{\,\pi \nu\,}{n} \cdot
\ln\csc\frac{\pi \nu}{n}  \,, \label{i56vg23543} 
\end{eqnarray}
$k=1,2,\ldots,n-1$\,, which may be obtained without much difficulty from the definition of $C_n(\nu)$.
For instance, the recurrence relationship \eqref{3984ycbf34y} is obtained as follows:
\begin{eqnarray}
&&\displaystyle \notag
C_n(\nu+1)\,=\sum_{l=1}^{n-1} \cos\!\left(\frac{2\pi \nu l}{n}+\frac{2\pi l}{n}\right)\csc\frac{\,\pi l\,}{n} =
\sum_{l=1}^{n-1} \cos\frac{2\pi \nu l}{n}\cdot\csc\frac{\,\pi l\,}{n}\cdot\cos\frac{2\pi l}{n} - \\[3.5mm]
&&\displaystyle\notag\quad
- \sum_{l=1}^{n-1} \sin\frac{2\pi \nu l}{n}\cdot\csc\frac{\,\pi l\,}{n}\cdot\sin\frac{2\pi l}{n} =
\sum_{l=1}^{n-1} \cos\frac{2\pi \nu l}{n}\cdot\csc\frac{\,\pi l\,}{n}\cdot\left(\!1-2\sin^2\frac{\pi l}{n}\right) - \\[3.5mm]
&&\displaystyle\notag\quad
- \, 2\!\sum_{l=1}^{n-1} \sin\frac{2\pi \nu l}{n}\cdot\cos\frac{\pi l}{n} 
=\, C_n(\nu) \,- \,2\!\sum_{l=1}^{n-1} \sin\frac{\,(2\nu+1)\pi l\,}{n}  \\[3.5mm]
&&\displaystyle\qquad\qquad\qquad\qquad\qquad\quad\; =\,C_n(\nu)  \,-\,2\ctg\frac{\,(2\nu+1)\pi\,}{2n}\,,
\end{eqnarray}
where at the final step we accounted for a well--known result
\be\label{0imxd143i9m}
\sum_{l=1}^{n-1} \sin\frac{\,\pi r l\,}{n}\, =
\begin{cases}
\ctg\dfrac{\,\pi r\,}{2n} \,,\qquad  	& r=1,3,5,\ldots\\[1mm]
0\,,			\qquad 			& r=2,4,6,\ldots
\end{cases}
\ee
Using repeatedly \eqref{3984ycbf34y} for $C_n(\nu+2)$, $C_n(\nu+3)$ and so on, we obtain \eqref{jkhbwc8qfc032}.
Since the sum is commutative, we also have \eqref{834cn384d3f}. The duplication formula \eqref{iwehnc83} is obtained from \eqref{jkhbwc8qfc032}
by setting $\kappa=\nu$. Identities \eqref{4938cn349y} \emph{et seq.} follow from various summation and orthogonality properties of the cosine.
For instance, the second of properties \eqref{4938cn349y} is obtained by means of the following orthogonality formula
\be\notag
\sum_{\nu=0}^{n-1} \cos\frac{2\pi \nu k}{n}\cdot \cos\frac{2\pi \nu l}{n}\,=\,\frac{n}{\,2\,}\Big\{\delta_{k,l}+\delta_{k,n-l}\Big\},
\ee
where both discrete variables $k$ and $l$ run through $1$ to $n-1$, as well as with the help of \eqref{89ybwetv3b}.
The summation formulae involving the gamma and the digamma functions are obtained as follows.
Gauss' digamma theorem states that
\be\label{9urmc039}
\Psi \biggl(\!\frac{l}{n} \vphantom{\frac{1}{2}} \!\biggr) =\,-\gamma-\ln2n-\frac{\pi}{2}\ctg\frac{\,\pi l\,}{n} 
+ \sum_{\nu=1}^{n-1} \cos\frac{\,2\pi \nu l \,}{n} \cdot\ln\sin\frac{\pi \nu}{n}\,,
\ee
where $l=1, 2,\ldots, n-1 $, $n=2,3,4,\ldots\,$, see e.g.~\cite[Eqs.~(B.4b)]{iaroslav_07} or 
\cite[Vol.~I, Sec.~1.7.3, Eq.~(29)]{bateman_01}. Therefore
\be
\begin{array}{ll}
\displaystyle
\sum_{\nu=1}^{n-1}C_n(\nu)  & \displaystyle \!\ln\sin\frac{\pi \nu}{n}  =
\sum_{l=1}^{n-1} \csc\frac{\,\pi l\,}{n} \sum_{\nu=1}^{n-1} \cos\frac{2\pi \nu l}{n} \cdot\ln\sin\frac{\pi \nu}{n} \\[6mm]
& \displaystyle =\sum_{l=1}^{n-1}  \left\{\Psi \biggl(\!\frac{l}{n} \vphantom{\frac{1}{2}} \!\biggr)  
+\gamma+\ln2n+\frac{\pi}{2}\ctg\frac{\,\pi l\,}{n} \right\}\csc\frac{\,\pi l\,}{n} \, =
\end{array}
\ee
\be
\begin{array}{ll}
\displaystyle
& \displaystyle = \left(\gamma+\ln2n\right) S_n \, + 
\sum_{l=1}^{n-1} \Psi \biggl(\!\frac{l}{n} \vphantom{\frac{1}{2}} \!\biggr)  \csc\frac{\,\pi l\,}{n} + 
\,\frac{\pi}{2}\sum_{l=1}^{n-1} \cos\frac{\,\pi l\,}{n} \cdot \csc^2\frac{\,\pi l\,}{n}\,.
\end{array}
\ee
Remarking that the latter sum vanishes by virtue of symmetry yields \eqref{iwytb877}.
Formula \eqref{iiytb8t3} is derived analogously by using Malmsten's variant of Gauss' digamma theorem \cite[Eqs.~(B.4c)]{iaroslav_07}
instead of \eqref{9urmc039}. Finally, identity \eqref{i56vg23543} follows 
from Gauss' digamma theorem \eqref{9urmc039} and \eqref{4ouedy}.
Note, by the way, that many properties are related to Watson's trigonometric sum.

\begin{figure}[!t]   
\centering
\includegraphics[width=0.8\textwidth]{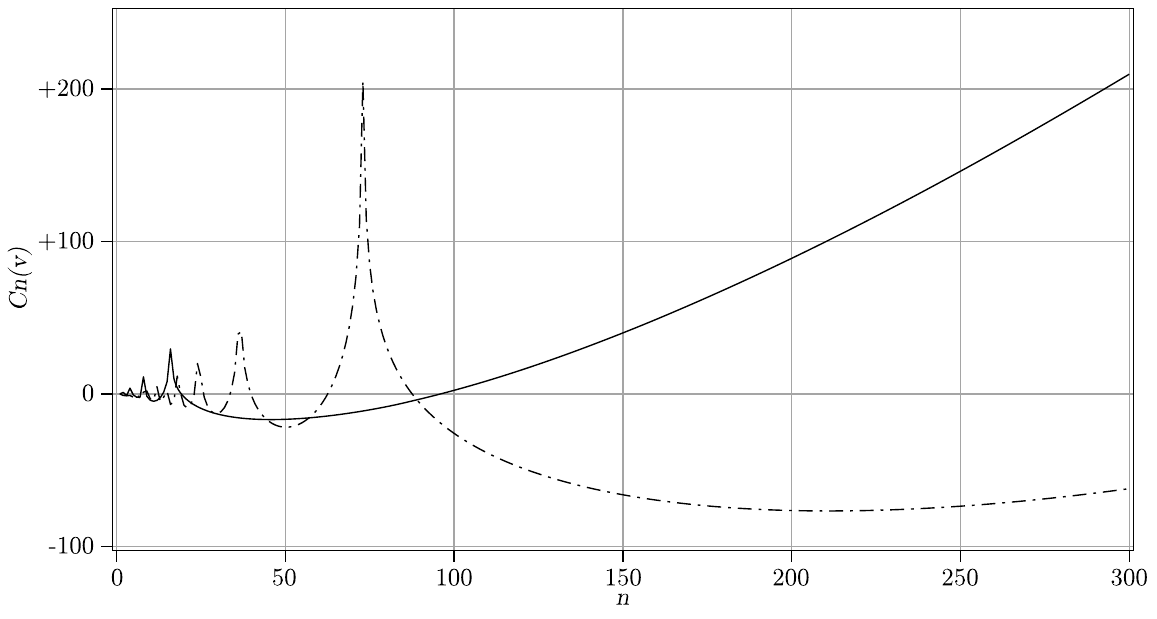}
\vspace{-0.5em}
\caption{The sum $C_{n}(\nu)$ as a function of $n$, where $n\in[2,300]$, for two different values of argument: $\nu=16$ (solid line)
and $\nu=73$ (dash--dotted  line). We deliberately take values of $\nu$, which lie outside the interval defined earlier,
in order to observe the overall behaviour of $C_{n}(\nu)$.
}
\label{g6hytbhfw}
\end{figure}

We conclude this section with two graphs of $C_n(\nu)$, depicted in Figs.~\ref{g6hytbhfw} and~\ref{g6hytbhfw2}.
The former figure displays the graph of $C_n(\nu)$ as a function of $n$ for two values of $\nu$.
It is quite remarkable to observe, see Fig.~\ref{g6hytbhfw}, that when $n<\nu$, as $n$ approaches the values having common divisors with $\nu$,
the sums $C_n(\nu)$ reaches a local maximum. 
It can also be clearly observed, see Fig.~\ref{g6hytbhfw2}, 
that $C_n(\nu)$ reaches the minimum when $\nu=\frac12n=150$, see property \eqref{98743cb38} and remark that $-C_{300}\approx-132$, 
and the maximum at the endpoints, at which $C_{300}(\nu)$ becomes equal to $S_{300}\approx1113$ (we do not show them on the graph, 
since they are too high).
The explanation of these interesting phenomena is given on p.~\pageref{iuo4ycn39v2} and follows from Theorem~\ref{iuo4ycn39}.

\section{Integral representation}
Below, we derive an important integral representation for the sum $C_n(\nu)$, which is useful for the
establishment of some further properties.
\begin{theorem}[Integral representations for $\bm{C_n(\nu)}$]\label{jh97yb876vb}
The sum $C_n(\nu)$, as defined in \eqref{984ycbn492v3}, may be represented via these integrals, 
containing the discrete Poisson kernel
\be\notag
\sum_{l=1}^{n-1} \cos\frac{2\pi \nu l}{n}\cdot\csc\frac{\,\pi l\,}{n} \,= \, \frac{\,2n\,}{\,\pi\,} 
\!\!\!\bigints\limits_{\!\!\!\!\!\!\!\!\!0}^{\;\;\;\;\;\;\,1} \!\! 
\frac{\,\big(1+x^n\big)\cos\dfrac{2\pi \nu}{n} - x^{n-1} - x \, }{1+x^n} \cdot
\frac{dx}{\,x^2-2x\cos\dfrac{2\pi \nu}{n}+1\,} \, ,
\ee

which may also be written in several alternative forms; for example,
\be\notag
\sum_{l=1}^{n-1} \cos\frac{2\pi \nu l}{n}\cdot\csc\frac{\,\pi l\,}{n} \,= \, \big(n-2 \nu\big)\ctg\frac{2\pi \nu}{n}
\,- \,\frac{\,n\,}{\,\pi\,} 
\!\!\!\bigints\limits_{\!\!\!\!\!\!\!\!\!0}^{\;\;\;\;\;\;\,\infty} \!\! 
\frac{\ch\left[x\!\left(\dfrac{n}{2}-1\right)\right] }{\,\ch\dfrac{xn}{2}\left(\ch x - \cos\dfrac{2\pi \nu}{n}\right)}\, dx \, .
\ee
\end{theorem}

\begin{figure}[!t]   
\centering
\includegraphics[width=0.8\textwidth]{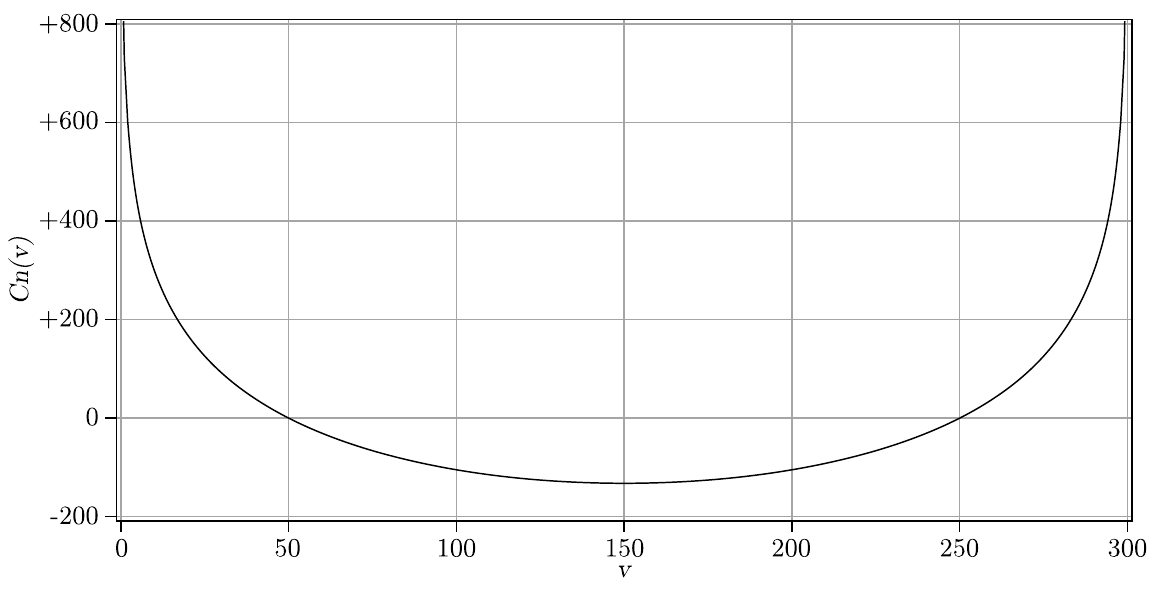}
\vspace{-0.5em}
\caption{The sum $C_{300}(\nu)$ as a function of $\nu$, where $\nu\in[0,300]$. We clearly observe that it
verifies the basic property $C_{n}(\nu)=C_{n}(n-\nu)$, see \eqref{39ufif}, and that it may take both positive and negative values. 
It is also visible that the minimum of $\left|C_{n}(\nu)\right|$ occurs at $\nu=50=\frac16n$ and $\nu=250=\frac56n$. More generally,
empirical studies show that at such points $C_{n}(\nu)$ is always very small, never equal to zero, but tends to zero when $n\to\infty$,
$n$ being a multiple of 6 (see also Corollary~\ref{lpo3sai} in Sect.~\ref{2093u0j}). 
}
\label{g6hytbhfw2}
\end{figure}

\begin{proof}
Let $x$ and $\varphi$ be two real variables such that $xe^{i\varphi}\neq1$. On summing the geometric progression
\be\notag
\sum_{l=1}^{n-1} x^l e^{i\varphi l}\,=\,\sum_{l=1}^{n-1} \big(x e^{i\varphi}\big)^l\,
=\,\frac{\, x^ne^{i\varphi n} - xe^{i\varphi}\,}{xe^{i\varphi}-1}\,, \qquad n=2,3,4,\ldots
\ee
and then on separating the real and the imaginary parts, we obtain
\begin{eqnarray}
&&\displaystyle \label{874db3478v1}
\sum_{l=1}^{n-1} x^l \cos\varphi l \,=\,\frac{\,x^{n+1}\cos\varphi(n-1)-x^n\cos\varphi n +x\cos\varphi-x^2\,}{x^2-2x\cos\varphi +1}
\end{eqnarray}
and
\begin{eqnarray}
&&\displaystyle \label{874db3478v2}
\sum_{l=1}^{n-1} x^l \sin\varphi l \,=\,\frac{\,x^{n+1}\sin\varphi(n-1)-x^n\sin\varphi n +x\sin\varphi\,}{x^2-2x\cos\varphi +1}
\end{eqnarray}
respectively. 
Consider now the following well--known result due to Euler:
\be\label{hd2893dh2}
\int\limits_0^\infty \!\frac{\,x^{p-1}}{\,1+x\,} \, dx\,=\,\pi\csc p\pi\,,\qquad 0<\Re{p}<1 \,,
\ee
see e.g.~\cite[\no 3.222-2]{gradstein_en}, \cite[\no 856.02]{dwigth_01_en}, \cite[p.~170 \& p.~172, \no 5.3.4.20-1]{mitrinovic_02}, 
\cite[p.~125, \no 878]{volkovyskii_01_eng}. 
Setting $\,p=l/n\,$, $n\in\mathbbm{N}\setminus\{1\}$, multiplying both sides by $\cos\varphi l$\,, putting $\,\varphi=2\pi \nu/n\equiv\theta\,$,
and then summing the result from $l=1$ to $l=n-1$, we have by virtue of \eqref{874db3478v1}
\be\label{9034dyh3y4r}
\begin{array}{ll}
\displaystyle
\sum_{l=1}^{n-1} \cos\theta l \cdot\csc\frac{\,\pi l\,}{n} \, & \displaystyle =\,\frac{\,1\,}{\pi}\!
\int\limits_0^\infty \! \frac{\,1}{\,x(1+x)\,} \sum_{l=1}^{n-1} x^\frac{l}{n} \cos\theta l  \, dx\, \\[8mm]
& \displaystyle
=\,\frac{\,1\,}{\pi}\! \int\limits_0^\infty 
\frac{\,x^{\frac1n+1}\cos\theta - x +x^{\frac1n}\cos\theta-x^{\frac2n}\,}{x^{\frac2n}-2x^{\frac1n}\cos\theta +1}
\cdot \frac{\,dx\,}{\,x(1+x)\,} \\[8mm]
& \displaystyle
=\,\frac{\,n\,}{\pi}\! \int\limits_0^\infty \!
\frac{\,\,y^{n}\cos\theta - y^{n-1} + \cos\theta - y\,}{ \,\big(1+y^n\big) \big(y^2-2y\cos\theta +1\big)\,}
\, dy\,  \\[8mm]
& \displaystyle
=\,\frac{\,2n\,}{\pi}\! \int\limits_0^1 \!
\frac{\,\,\big(1+y^n\big) \cos\theta - y^{n-1} - y\,}{\, \big(1+y^n\big) \big(y^2-2y\cos\theta +1\big)\,}
\, dy\, ,
\end{array}
\ee
where we made a change of variable $y=x^{\frac1n}$ and then split the interval of integration into two parts
$[0,1]$ and $[1,\infty)$, the integral over $[1,\infty)$ being equal to that over $[0,1]$.
We, thus, have arrived at the first result of the Theorem.

Making once again a change of variable $y=e^{-t}$ in the last integral in \eqref{9034dyh3y4r}, we get
\be\label{08u4nvc38}
\begin{array}{ll}
\displaystyle 
C_n(\nu)  & \displaystyle =\,\frac{\,2n\,}{\pi}\! \int\limits_0^1 \!
\frac{\,\big(1+y^n\big) \cos\theta - y^{n-1} - y\,}{\, \big(1+y^n\big) \big(y^2-2y\cos\theta +1\big)\,}\,dy \notag\\[8mm]
\displaystyle 
\: & \displaystyle =\,\frac{\,n\,}{\pi}\! \int\limits_0^\infty \!
\frac{\,\ch\frac12tn \cdot\cos\theta  - \ch\big[t\big(\frac12n-1\big)\big]\,}{\,\big(\ch t  - \cos\theta \big) 
\ch\frac12tn \,} \,dt\,.
\end{array}
\ee
The final step, of accounting for the elementary integral\footnote{See e.g.~\cite[\no 2.444-2]{gradstein_en} or \cite[\no 859.163]{dwigth_01_en}.}
\be\notag
\int\limits_0^\infty \frac{dt}{\,\ch t - \cos\varphi\,}\,=\,
\begin{cases}
\displaystyle \frac{\,\pi-\varphi\,}{\sin\varphi}\,,\qquad & 0<\varphi<2\pi\,, \quad \varphi\neq\pi \,,\\[3mm]
\displaystyle 1\,,\qquad & \varphi=\pi\,,
\end{cases}
\ee
evaluated at $\varphi=\theta$,
yields the second result of the Theorem. Note that both results of the Theorem hold only for integer $\nu$, although the above method can,
of course, be extended to the continuous values of $\nu$ as well, in which case the resulting expression will be more complicated.\footnote{The assumption
of discrete $\nu$ permits to greatly simplify our calculations in the second line of \eqref{9034dyh3y4r}.}
\end{proof}

\section{Series representations}
In this section we obtain several series representations for $C_n(\nu)$.
We begin with four alternative finite series representations. 

\begin{theorem}[Finite series representations for $\bm{C_n(\nu)}$]\label{iuo4ycn39}
The finite sum $C_n(\nu)$ and Watson's trigonometric sum $S_n$ are directly related to each other by means
of these four nonexhaustive formulae
\be\notag
\begin{array}{ll}
\displaystyle
\sum_{l=1}^{n-1} \cos\frac{2\pi \nu l}{n}\cdot\csc\frac{\,\pi l\,}{n}\:&\displaystyle= \, S_n \,-\,
2\sum_{l=1}^{\nu} \ctg\frac{\,\left(2l-1\right)\pi\,}{2n}\,,\\[7mm]
\displaystyle
\sum_{l=1}^{n-1} \cos\frac{2\pi \nu l}{n}\cdot\csc\frac{\,\pi l\,}{n}\:&\displaystyle= \, S_n \,-\,
2\sum_{l=1}^{n-1} \sin^2\frac{\pi \nu l}{n}\cdot\csc\frac{\,\pi l\,}{n}\,,\\[7mm]
\displaystyle
\sum_{l=1}^{n-1} \cos\frac{2\pi \nu l}{n}\cdot\csc\frac{\,\pi l\,}{n}\: &\displaystyle= \,
-\,S_n \,+\,
2\sum_{l=1}^{n-1} \cos^2\frac{\pi \nu l}{n}\cdot\csc\frac{\,\pi l\,}{n}\,,\\[7mm]
\displaystyle
\sum_{l=1}^{n-1} \cos\frac{2\pi \nu l}{n}\cdot\csc\frac{\,\pi l\,}{n}\:&\displaystyle= \,S_n  -\ctg\frac{\,\pi\,}{2n}+\ctg\frac{\,\pi\left(2\nu+1\right)\,}{2n}\,-\\[7mm]
&\displaystyle\qquad
-2\sum_{l=1}^{n-1} \ctg\frac{\,\pi l\,}{n}\cdot\sin\frac{\pi l \nu}{n}\cdot\sin\frac{\pi l \left(\nu+1\right)}{n}\,.
\end{array}
\ee
\end{theorem}

\begin{proof}
The first formula is deduced from property \eqref{jkhbwc8qfc032}. 
Setting $\nu=0$, using \eqref{9348yb394y} and then writing $\nu$ instead of $\kappa$,
yields this formula. The second and third equalities follow from the elementary trigonometric
identity $\cos2\alpha=2\cos^2\alpha-1=1-2\sin^2\alpha$. Finally, the fourth equality is obtained from the first one in the following manner. 
From \eqref{0imxd143i9m} it follows that the cotangent sum may be written in the following form
\be\label{89cy5un23}
\begin{array}{ll}
\displaystyle
&\displaystyle 
\sum_{l=1}^{\nu} \ctg\frac{\,\left(2l-1\right)\pi\,}{2n} \, =\,\sum_{k=1}^{n-1}\sum_{l=1}^{\nu} 
\sin\!\left(\frac{2\pi k l}{n}-\frac{\pi k}{n}\right)=
\\[8mm]
&\displaystyle\qquad\qquad\qquad
= \sum_{k=1}^{n-1}\left\{\!
\left[-\frac{\cos\!\left(\frac{2\pi k \nu}{n}+\frac{\pi k}{n}\right)}{2\sin\frac{\,\pi k\,}{n}}+\frac12\ctg\frac{\,\pi k\,}{n}\right]
\!\times\cos\frac{\,\pi k\,}{n} \, - \right.\\[8mm]
&\displaystyle\qquad\qquad\qquad\qquad\qquad\qquad
\left. - 
\left[\frac{\sin\!\left(\frac{2\pi k \nu}{n}+\frac{\pi k}{n}\right)}{2\sin\frac{\,\pi k\,}{n}}-\frac12\right] \!\times
\sin\frac{\,\pi k\,}{n} \right\}.
\end{array}
\ee
by virtue of \eqref{874db3478v1} and \eqref{874db3478v2} evaluated at $x=1$ and $\varphi=2\pi k/n$.
Then, employing twice \eqref{0imxd143i9m}, we obtain
\be\notag
\sum_{k=1}^{n-1}\sin\frac{\,\pi k\,}{n}\,=\,\ctg\frac{\,\pi\,}{2n} \qquad \text{and}\qquad
\sum_{k=1}^{n-1}\sin\frac{\,\pi k \left(2\nu+1\right)\,}{n}\,=\,\ctg\frac{\,\left(2\nu+1\right)\pi\,}{2n} \,.
\ee
Inserting latter formulae into \eqref{89cy5un23} yields, after some algebra, the fourth formula of the Theorem.
\end{proof}

\label{iuo4ycn39v2}
Last three formulae obtained in the preceding theorem are, in some sense, similar. As to the first formula
$\,C_n(\nu)\,= \,S_n - 2\!\sum\limits_{l=1}^{\nu} \! \ctg\frac{\,(2l-1)\pi\,}{2n}\,$, it is clearly of different type and
merits to be briefly discussed. 
First of all, it immediately implies that $C_n(\nu)$ cannot be greater than $S_n$ and lesser than $-C_n$,
properties that we already gave in Section~\ref{h09387rxhxdws}, and which follows from the fact 
that $\sum\ctg\frac{\,(2l-1)\pi\,}{2n}$ increases while $\nu$ remains below $\frac12n$, decreases when $\nu>\frac12n$
(when $l>\frac12n$, the $l$th term sums with the $(n-l)$th term, which has the same magnitude and the opposite sign, leading
thus to the overall decrease of $\sum\ctg\frac{\,(2l-1)\pi\,}{2n}$), and vanishes as $\nu$ reaches $n$, see e.g.~Fig.~\ref{g6hytbhfw2}.
It also explains why as long as $n<\nu$, the sum $C_n(\nu)$ reaches local maxima as $n$ approaches common divisors with $\nu$ 
(see Fig.~\ref{g6hytbhfw}). This is due to the sum of cotangents, which not only vanishes at $n=\nu$, but as long as $n<\nu$ is also relatively small each time when $\gcd(n,\nu)>1$. 
Furthermore, as noted before, $S_n$ is very well investigated; the study of $C_n(\nu)$ may, therefore, be reduced 
to that of $\sum\ctg\frac{\,(2l-1)\pi\,}{2n}$. Note, however, that the investigation of the cotangent sums often faces 
considerable difficulties;\footnote{Despite the fact that they have been regularly studied at least since Euler's time, 
see e.g.~\cite[Sections 1.1 \& 4]{iaroslav_18}.} furthermore,
some of such sums even appear to be directly associated to the study of the Riemann
hypothesis \cite{beck_02,bettin_01,bettin_02,byrne_01,cvijovic_00,cvijovic_03,cvijovic_05,derevyanko_01,dieter_01,ejsmonta_01,folsom_01,goubi_01,lewis_01,
maier_01,rassias_0,sofo_01}.\\

We come to obtain now an infinite series representation for the function $C_n(\nu)$,
which is useful for the derivation of the asymptotic expansion of $C_n(\nu)$ at large $n$.
To derive it, we first need to prove the following Lemma.

\begin{lemma}\label{8743ctb23478}
If $\,\Re(\alpha+b - \beta)>0\,$, the following improper integral converges and may be evaluated 
via a combination of four digamma functions
\be\notag
\begin{array}{ll}
\displaystyle 
\int\limits_0^\infty \! e^{-\alpha x} \, \frac{\ch\beta x}{\ch bx}\,dx \,=\, \frac{1}{4b}&\displaystyle\!
\left\{\!\Psi\left(\frac34+\frac{\alpha+\beta}{4b}\right)-\Psi\left(\frac14+\frac{\alpha-\beta}{4b}\right) +\right.\\[7mm]
&\displaystyle \qquad \qquad \left.
+\,\Psi\left(\frac34+\frac{\alpha-\beta}{4b}\right)-\Psi\left(\frac14+\frac{\alpha+\beta}{4b}\right)\!\right\}.
 \end{array}
\ee
\end{lemma}

\begin{proof}
Consider the following integral
\be\notag
\int\limits_0^\infty \! e^{-\alpha x} \, \frac{\ch\beta x}{\ch bx}\,dx \,.
\ee
Multiplying the numerator and the denominator of the integrand by $2\sh bx$ yields
\be\notag
\begin{array}{ll}
\displaystyle 
\int\limits_0^\infty \! e^{-\alpha x} \, \frac{\ch\beta x}{\ch bx}\,dx &\displaystyle \,=\,
2\!\int\limits_0^\infty \! e^{-\alpha x} \, \frac{\, \sh bx  \ch\beta x\,}{\,2\ch bx  \sh bx\,}\,dx \\[7mm]
&\displaystyle =\,2\!\int\limits_0^\infty \! e^{-\alpha x} \, \frac{\,\sh[x(b+\beta)] + \sh[x(b-\beta)]\,}{\,\sh 2bx\,}\,dx \,.
\end{array}
\ee
Using twice the well--known formula
\be\notag
\int\limits_0^\infty \! e^{-\alpha x} \, \frac{\sh\mu x}{\sh mx}\,dx 
\,=\, \frac{1}{2m}\left\{\Psi\!\left(\frac12+\frac{\alpha+\mu}{2m}\right)-
\Psi\!\left(\frac12+\frac{\alpha-\mu}{2m}\right)\!\right\}\,, 
\ee
which holds for $\Re(\alpha+m - \mu)>0$, 
see e.g.~\cite[\no 3.541-2]{gradstein_en}, \cite[Vol.~I, Sec.~1.7.2, Eqs.~(14)--(15)]{bateman_01},
with $m=2b$ and $\mu=b\pm\beta$ respectively,
we immediately arrive at the announced result.
\end{proof}

\begin{theorem}[Infinite series representations for $\bm{C_n(\nu)}$]\label{iuoyh9}
The finite sum $C_n(\nu)$, defined by \eqref{984ycbn492v3}, admits the following series representation
\begin{eqnarray}
&&\displaystyle 
\!\!\!
\sum_{l=1}^{n-1} \cos\frac{2\pi \nu l}{n}\cdot\csc\frac{\,\pi l\,}{n} =  \frac{\,2n\,}{\,\pi\,}\ln\!\left(2\sin\frac{\pi\nu}{n}\right)
- \frac{2}{\pi}\left\{\Psi\!\left(\frac{2}{n}\right) - \Psi\!\left(\frac{1}{n}\right) \!\right\}
-\frac{2}{\pi}\csc\frac{2\pi\nu}{n} \times \notag\\[2mm]
&&\displaystyle \qquad
\times\sum_{l=2}^{\infty}\left\{\!\Psi\!\left(\frac{l-1}{2n}\right) - \Psi\!\left(\frac{l+1}{2n}\right) 
-\Psi\!\left(\frac{l-1}{n}\right) + \Psi\!\left(\frac{l+1}{n}\right)  \!\right\}\sin\frac{2\pi \nu l}{n} \,.\label{iuoyh9formula}
\end{eqnarray}
where the last infinite series converges at the same rate as $\,\sum l^{-2}\sin(2\pi \nu l/n) \,$.
\end{theorem}

\begin{proof}
Consider the second formula of Theorem~\ref{jh97yb876vb} and let denote $\,\theta\equiv2\pi \nu/n\,$ for the purpose of brevity.
By expanding $\,(\ch x - \cos\theta)^{-1}$ into the uniformly convergent series
\be\notag
\frac{1}{\,\ch x - \cos\theta\,}  \,=\,\frac{2}{\,\sin\theta\,}\!
\sum_{l=1}^{\infty} e^{-x l}\sin\theta l\,, \qquad\, \Re x>0\,,
\ee
and by using Lemma~\ref{8743ctb23478}, we see at once that
\be\label{0834ucn30}
\begin{array}{ll}
\displaystyle 
&\displaystyle C_n(\nu) \,= \, \big(n-2 \nu\big)\ctg\theta
\,- \,\frac{\,2n\,}{\,\pi\sin\theta\,} 
 \sum_{l=1}^{\infty} \sin\theta l \! \int\limits_0^\infty \! e^{-x l} \,
\frac{\,\ch\big[x\big(\frac12n-1\big)\big]\,}{\,\ch\frac12xn \,} \,dx\,=\\[8mm]
&\displaystyle\quad
= \, \big(n-2 \nu\big)\ctg\theta
\,- \,\frac{\,1\,}{\,\pi\sin\theta\,} 
 \sum_{l=1}^{\infty} \sin\theta l \left\{\Psi\!\left(1+\frac{l-1}{2n}\right) - \Psi\!\left(\frac{l+1}{2n}\right) 
 +  \right.\\[8mm]
 &\displaystyle\quad\;\;\left.
 + \,\Psi\!\left(\frac12+\frac{l+1}{2n}\right)  - \, \Psi\!\left(\frac12+\frac{l-1}{2n}\right)
 \!\right\} =\, \big(n-2 \nu\big)\ctg\theta \, - \,\frac{\,1\,}{\,\pi\,} \times \\[8mm]
\displaystyle 
&\displaystyle\qquad
 \times\left\{2\ln2+ \Psi\!\left(\frac12+\frac{1}{n}\right)  - \Psi\!\left(\frac{1}{n}\right)\!\right\}
 - \,\frac{\,1\,}{\,\pi\sin\theta\,} \sum_{l=2}^{\infty}\sin\theta l\left\{\frac{2n}{l-1}\,  +\right.\\[8mm]
 &\displaystyle\qquad \left.
 +\Psi\!\left(\frac{l-1}{2n}\right) - \Psi\!\left(\frac{l+1}{2n}\right) 
 + \Psi\!\left(\frac12+\frac{l+1}{2n}\right)  - \Psi\!\left(\frac12+\frac{l-1}{2n}\right)  \!\right\}.
\end{array}
\ee
Using two well--known Fourier series in order to evaluate
\be\label{675v67v}
\begin{array}{ll}
\displaystyle
\sum_{l=2}^{\infty}\frac{\sin\theta l}{\,l-1\,}&\displaystyle \,=\,\cos\theta \sum_{l=1}^{\infty} \frac{\,\sin\theta l\,}{l} \,+
\,\sin\theta \sum_{l=1}^{\infty} \frac{\,\cos\theta l\,}{l}\\[8mm]
&\displaystyle \,=\,\frac{\,\pi-\theta\,}{2}\cos\theta \,
- \,\sin\theta\cdot\ln\!\left(2\sin\frac\theta2\right)\,,
\end{array}
\ee
and employing thrice the duplication formula for the digamma function
\be\notag
\begin{array}{l}
\displaystyle 
\Psi\!\left(\frac12+\frac{1}{n}\right)  - \Psi\!\left(\frac{1}{n}\right)=
2\Psi\!\left(\frac{2}{n}\right)  - 2\Psi\!\left(\frac{1}{n}\right) - 2\ln2 \\[8mm]
\displaystyle 
 \Psi\!\left(\frac{l-1}{2n}\right) -\Psi\!\left(\frac12+\frac{l-1}{2n}\right) =
2\Psi\!\left(\frac{l-1}{2n}\right) - 2\Psi\!\left(\frac{l-1}{n}\right) +2\ln2\\[8mm]
\displaystyle 
\Psi\!\left(\frac12+\frac{l+1}{2n}\right)  - \Psi\!\left(\frac{l+1}{2n}\right) =
-2\Psi\!\left(\frac{l+1}{2n}\right) + 2\Psi\!\left(\frac{l+1}{n}\right) -2\ln2\,,
\end{array}
\ee
the last expression in \eqref{0834ucn30} immediately reduces to \eqref{iuoyh9formula}.

Finally, in order to study the behaviour of the general term of series \eqref{iuoyh9formula} at large index $l$, we use 
the Stirling formula 
\be\label{lk2093mffmnjw}
\Psi(x) \, =\,\ln x \,-\, \frac{1}{\,2x\,} \,-\, \frac{1}{\,2\,}\!\sum_{r=1}^{N-1} \frac{B_{2r}}{\,r\,x^{2r}\,}
\,-\, \frac{\lambda\,B_{2N}}{\,2Nx^{2N}\,}\,, \qquad x>0\,, 
\ee
where as usually $0<\lambda<1\,$ and $N=2,3,4,\ldots$, $N<\infty$.
For a sufficiently large $l$, we, therefore, have:
\be\notag
\Psi\!\left(\frac{l-1}{2n}\right) - \Psi\!\left(\frac{l+1}{2n}\right) 
-\Psi\!\left(\frac{l-1}{n}\right) + \Psi\!\left(\frac{l+1}{n}\right)
\,\sim\,-\frac{n}{\,l^2\,}
\ee
Thus, series \eqref{iuoyh9formula} converges slightly better than Euler's series $\,\sum l^{-2}$.
\end{proof}

\section{Preliminary asymptotic studies for large $n$}\label{4c312ct3q4rgfwa}
In this part we study the asymptotic behaviour of $C_n(\nu)$ at large $n.$ The results 
are based, mostly, on the series representations obtained earlier in Theorem~\ref{iuoyh9} 

\begin{theorem}[Almost asymptotic expansion of $\bm{C_n(\nu)}$]\label{ejbhw08}
For any $n=2,3,4,\ldots\,$ and $N=2,3,4,\ldots\,$, $N<\infty\,$, the sum $C_n(\nu)$
admits the following expansion  
\begin{eqnarray}
\displaystyle 
\sum_{l=1}^{n-1} \! &\displaystyle\cos\frac{2\pi \nu l}{n}\cdot\csc\frac{\,\pi l\,}{n}  \,=\,  -\frac{\,2n\,}{\,\pi\,}
\ln\!\left(2\sin\frac{\pi\nu}{n}\right) +\, 2\!\sum_{r=1}^{N-1} \frac{\,\big(1-2^{1-2r}\big)\pi^{2r-1} B_{2r}\,}{\,(2r)!\, n^{2r-1}\,} 
\times \notag\\[3.5mm]
&\displaystyle \times \mathscr{H}_{2r-1}(\nu/n) \,
+\, 2\lambda\frac{\,\big(1-2^{1-2N}\big)\pi^{2N-1} B_{2N}\,}{\, (2N)!\, n^{2N-1}\,} 
\mathscr{H}_{2N-1}(\nu/n) \,,\label{8943ycn9438}
\end{eqnarray}
$0<\lambda<1$, where we denoted for the sake of brevity 
\be\label{93402s}
\begin{array}{lll}
\displaystyle
\mathscr{H}_{2r-1}(\nu/n) \,\equiv\,\left.\frac{d^{2r-1}}{d\varphi^{2r-1}}\, \ctg\varphi \right|_{\varphi=\pi\nu/n} &\displaystyle =
-\frac{1}{\,\pi^{2r}\,}\left\{\Psi_{2r-1}\left(\frac{\nu}{n}\right) + \Psi_{2r-1}\left(1-\frac{\nu}{n}\right) \!\right\}\\[8mm]
&\displaystyle = (-1)^r \frac{(2n)^{2r-1}}{r}\sum_{s=1}^n B_{2r}\!\left(\frac{s}{n}\right) \cos\frac{\,2\pi s \, \nu\,}{n}
\end{array}
\ee
where $B_{r}(x)$ are the Bernoulli polynomials and where all $\mathscr{H}_{2r-1}(\nu/n)\,$ are negative.
\end{theorem}

Remark that the above theorem provides the expansion of $C_n(\nu)$ via the derivatives of the cotangent;
in fact, not only the tail contains such derivaties, but also the dominant term, which is is actually, up to some coeffcients, 
the antiderivative of the cotangent (the only term which does not contain the derivatives of the cotangent is $-\frac{2n\ln2}{\pi}$).
The above formula leads to several important corollaries and has also several useful applications. 
In particular, under some conditions, on may readily deduce from it several asymptotic formulae
for $C_n(\nu)$ and $C_n$ at large $n$.

\begin{corollary}[Asymptotic representation of $\bm{C_n(\nu)}$ at large $\bm{n}$]\label{ejbhw08v2}
At large $n$ the following asymptotic representation holds for the sum $C_n(\nu)$
\be\notag
\sum_{l=1}^{n-1} \cos\frac{2\pi \nu l}{n}\cdot\csc\frac{\,\pi l\,}{n}\,\sim\, 
-\frac{\,2n\,}{\,\pi\,}\ln\!\left(2\sin\frac{\pi\nu}{n}\right)\,, 
\ee
where $\nu\neq\frac16n$ and $\nu\neq\frac56n$.
\end{corollary}

\begin{corollary}[Complete asymptotics of $\bm{C_n}$ at large $\bm{n}$]\label{9784ryb}
Let $N=2,3,4,\ldots\,$, $N<\infty\,$. Then, for $n=2,4,6,\ldots\,$ the sum $C_n$, admits the following expansion
\be\notag
\begin{array}{ll}
\displaystyle
\sum_{l=1}^{n-1} (-1)^{l+1} \csc\frac{\,\pi l\,}{n}\, = \,\frac{\,2n\ln2\,}{\,\pi\,}\:
+  & \displaystyle 2\sum_{r=1}^{N-1} \frac{\,(-1)^{r+1}\big(2^{2r-1}-1\big)
\big(2^{2r}-1\big)\pi^{2r-1} B^2_{2r}\,}{\,r\, (2r)!\, n^{2r-1}\,}  \, +  \, \\[8mm]
\displaystyle\: & \displaystyle
 \, +  \, O\big(n^{1-2N}\big)\,.
\end{array}
\ee
which at $n\to\infty$ becomes its complete (or full) asymptotics.
Writing down first few terms, we have
\be\notag
\sum_{l=1}^{n-1} (-1)^{l+1} \csc\frac{\,\pi l\,}{n}\,=\,\frac{\,2n\ln2\,}{\,\pi\,}
+\,\frac{\pi}{\,12\, n\,} \,-\,\frac{7\pi^3}{\,1440\,n^3\,} \,+\,\frac{31\pi^5}{\,30\,240\,n^5\,}\,
-\,\frac{2159\pi^7}{\,4\,838\,400\,n^7\,}\,+\,\ldots
\ee
If, in contrast, $n$ is odd, then simply
\be\notag
\sum_{l=1}^{n-1} (-1)^{l+1} \csc\frac{\,\pi l\,}{n}\, = \,0\,.
\ee
\end{corollary}

\begin{remark}\label{9843yndc3894}
The above asymptotic expansion may be compared to that of 
\be\notag
\sum_{l=1}^{n-1}  \csc\frac{\,\pi l\,}{n}\,  =  \,\frac{\,2n\,}{\pi}\!\left(\ln\frac{2n}{\pi}+\gamma \right)
-\frac{2}{\,\pi\,}\!\sum_{r=1}^{N-1} \frac{\,(-1)^{r+1} \big(2^{2r-1}-1\big)\pi^{2r} B^2_{2r}\,}{\,r\, (2r)!\, n^{2r-1}\,}
+ O\big(n^{1-2N}\big) \,,
\ee
which is probably due to Watson, see \cite[Theorems 6a,b]{iaroslav_18}, and which also contains the square of the Bernoulli numbers.
It is also interesting that unlike $\,\sum\csc(\pi l/n)$ the asymptotic expansion of $\,\sum(-1)^{l+1}\csc(\pi l/n)$ 
does not contain Euler's constant.
\end{remark}

\begin{proof}
Consider \eqref{874db3478v1}--\eqref{874db3478v2} at $x=1$ 
\begin{eqnarray}
&&\displaystyle \label{874db3478v5}
\sum_{l=1}^{n-1}  \cos\varphi l \,=\,\frac{\sin\big(n\varphi-\frac12\varphi\big)\,}{\,2\sin\frac12\varphi\,}\,
-\,\frac{1}{\,2\,}\,,\\[3mm]
&&\displaystyle \label{874db3478v6}
\sum_{l=1}^{n-1} \sin\varphi l \,=\,-\frac{\cos\big(n\varphi-\frac12\varphi\big)\,}{\,2\sin\frac12\varphi\,}
\,+\,\frac{1}{\,2\,}\ctg\frac\varphi2\,.
\end{eqnarray}
As $n$ tends to infinity, both series diverge, but if $\varphi$ is not 
congruent to $0 \pmod{2\pi}$, they still remain Ces\`aro summable
\be\label{874db3478v3}
\sum_{l=1}^{\infty} \cos\varphi l \,=\,-\,\frac{1}{\,2\,}\qquad(\text{C},1)\,,\qquad\quad\qquad
\sum_{l=1}^{\infty}\sin\varphi l \,=\,\frac{1}{\,2\,}\ctg\frac\varphi2\qquad(\text{C},1)\,,
\ee
since
\be\notag
\sum_{n=1}^N \sin\big(n\varphi-\tfrac12\varphi\big)\, =\,O(1)
\qquad \text{and} \qquad
\sum_{n=1}^N \cos\big(n\varphi-\tfrac12\varphi\big)\, =\,O(1)\,,
\ee
as $N\to\infty$, and hence 
\be\notag
\lim_{N\to\infty} \left\{\frac1N\sum_{n=1}^N \frac{\sin\big(n\varphi-\frac12\varphi\big)\,}{\,2\sin\frac12\varphi\,} \right\} =0
\quad \text{and} \quad
\lim_{N\to\infty} \left\{\frac1N\sum_{n=1}^N \frac{\cos\big(n\varphi-\frac12\varphi\big)\,}{\,2\sin\frac12\varphi\,} \right\} =0\,
\ee
respectively. Moreover, formulae \eqref{874db3478v3} remain also true 
if using other \emph{regular summation methods}, such as, for example, Euler summations $(\mathfrak{E})$ and $(\text{E},1)$, Abel summation $(\text{A})$,
Borel summation methods $(\text{B})$ and $(\text{B'})$, etc.\footnote{Readers interested in a more deep study 
of the divergent series, are kindly invited to refer to monograph \cite{hardy_02}. Note that the use of divergent
series for the derivation of asymptotic series is a frequent practice \cite{copson_01,dingle_01,erdelyi_01,evgrafov_03_eng,olver_01}.}

Let us now examine the series expansion for the digamma function given in \cite[Vol.~I, Sec.~1.17, Eq.~(5)]{bateman_01}.
This series possesses the property that the error, due to stopping at any term, is numerically less than the first term neglected. 
This means that the same series may also be written in the following form
\be\label{8976bt68}
\Psi(x)\,=\,-\frac1x-\gamma+\sum_{m=2}^{M-1} (-1)^m x^{m-1}\zeta(m) +  \lambda (-1)^M x^{M-1} \zeta(M)\,,
\ee
where $0<\lambda<1\,$, $M=3,4,5,\ldots\,$, and which has the advantage to hold for any $x>0$, while \cite[Vol.~I, Sec.~1.17, Eq.~(5)]{bateman_01} holds only in the unit disc.
Using this expansion, as well as \eqref{874db3478v3} and their formal derivatives with respect to $\varphi$, the second line
of \eqref{iuoyh9formula} becomes:
\begin{eqnarray}
\displaystyle && \displaystyle
\sum_{l=2}^{\infty}\left\{\Psi\!\left(\frac{l-1}{2n}\right) - \Psi\!\left(\frac{l+1}{2n}\right) 
-\Psi\!\left(\frac{l-1}{n}\right) + \Psi\!\left(\frac{l+1}{n}\right)  \!\right\}\sin\theta l \, =\, \notag\\[3.5mm]
\displaystyle && \displaystyle\qquad
=-2n\sum_{l=2}^{\infty} \frac{\,\sin\theta l\,}{\,l^2-1\,} \, - \,
\sin\theta\cdot\!\!\! \sum_{m=2}^{2N-2} (-1)^{m} \big(1-2^{1-m}\big)\zeta(m) \times \label{khgbitbi}\\[3.5mm]
&& \displaystyle \qquad \qquad \qquad \qquad 
\times \frac{\,2\mathscr{C}_{m-1}(\theta)+2^{m-1}\,}{n^{m-1} } \, +\, R_N(\lambda,n,\theta)\,,\notag 
\end{eqnarray}
where again for brevity we put $\,\theta\equiv2\pi \nu/n$, where
\be\label{iwauce0udf}
\mathscr{C}_{m-1}(\theta)\,=\,
\begin{cases}
\displaystyle 0\,, \qquad & m=2r-1\,, \quad r\in\mathbbm{N}\,\\[3mm]
\displaystyle \frac{\,(-1)^{\frac{m}{2}-1}\,}{2}\left(\ctg\frac\varphi2\right)^{(m-1)}_{\varphi=\theta}\,, \qquad & m=2r\,, \quad r\in\mathbbm{N}
\end{cases}
\ee
and where $R_N(\lambda, n,\theta)$ stands for the remainder.
Proceeding with the first sum similarly to \eqref{675v67v}, one may easily show that
\be\label{9834yc20398}
\sum_{l=2}^{\infty} \frac{\,\sin\theta l\,}{\,l^2-1\,} \, =\,\left[\frac14 - \ln\!\left(\!2\sin\frac\theta2\right)\right]\sin\theta\,.
\ee
Furthermore
\be\label{998yncv93}
\left(\ctg\frac\varphi2\right)^{(2r-1)}_{\varphi=\theta} = \left.\frac{1}{\,2^{2r-1}\,} \cdot
\frac{d^{2r-1} \ctg\varphi}{d\varphi^{2r-1}}\right|_{\varphi=\frac12\theta}\equiv\,
\frac{\mathscr{H}_{2r-1}(\nu/n)}{\,2^{2r-1}\,}\,.
\ee
Now, employing again \eqref{8976bt68}, we see at once that a part of the last sum in \eqref{khgbitbi}
reduces to
\begin{eqnarray}
\displaystyle \: & \displaystyle
\sum_{m=2}^{2N-2} (-1)^{m} \big(1-2^{1-m}\big)\zeta(m) n^{1-m} 2^{m-1}  = 
\sum_{m=2}^{2N-2} (-1)^{m} \big(2^{m-1}-1\big)\zeta(m) n^{1-m}  =\notag\\[3mm]
& \displaystyle  =\, \Psi\!\left(\frac{2}{n}\right) - \Psi\!\left(\frac{1}{n}\right)-\,\frac{n}{2}\, +\,
\lambda\big(2^{2N-2}-1\big)\zeta(2N-1) n^{2-2N} \,,\label{98034ycn}
\end{eqnarray}
$0<\lambda<1$.
Substituting \eqref{998yncv93} into \eqref{iwauce0udf} and \eqref{9834yc20398} into \eqref{khgbitbi}, as well 
as accounting for \eqref{98034ycn}, formula \eqref{iuoyh9formula} from Theorem~\ref{iuoyh9} becomes
\begin{eqnarray}
\displaystyle 
C_n(\nu)
&\displaystyle =\,  -\frac{\,2n\,}{\,\pi\,}\ln\!\left(\!2\sin\frac\theta2\right)
+ \,\frac{2}{\,\pi\,}\!\sum_{r=1}^{N-1} \frac{\,(-1)^{r-1}\big(1-2^{1-2r}\big)\zeta(2r)\,}{\, (2n)^{2r-1}\,} 
\mathscr{H}_{2r-1}(\nu/n)  \, +\notag \\[4mm]
&\displaystyle \qquad
+\, 2\lambda \frac{\,(-1)^{N-1}\big(1-2^{1-2N}\big)\zeta(2N)\,}{\, \pi (2n)^{2N-1}\,} 
\mathscr{H}_{2N-1}(\nu/n) \,,\label{392u8rcm2190}
\end{eqnarray}
where again $0<\lambda<1$.
Interestingly, the latter result depends on $2n$ rather than on $n$.
Finally, using the famous result
\be\label{894yrnbcssw}
\zeta(2r)\,=\,\frac{(-1)^{r+1}(2\pi)^{2r} B_{2r}}{\,2\,(2r)!\,}\,,\qquad
r\in\mathbbm{N}\,,
\ee
established by Euler in the first half of the XVIIIth century,
we immediately arrive at the expansion announced in Theorem~\ref{ejbhw08}. 
The second representation of $\mathscr{H}_{2r-1}(\nu/n)$ in \eqref{93402s}, that via two polygamma functions, is obtained by differentiating 
$(2r-1)$ times the reflection formula of the digamma function $\pi\ctg\varphi=\Psi(1-\varphi/\pi)-\Psi(\varphi/\pi)$
with respect to $\varphi$, and then by setting $\varphi=\theta/2=\pi\nu/n$. The third representation of $\mathscr{H}_{2r-1}(\nu/n)$
directly follows from the relationship between the derivatives of the cotangent at rational multiples of $\pi$ and
the Bernoulli polynomials, see e.g.~\cite[p.~218]{cvijovic_04}.

Now, retaining only the dominant term in the asymptotic expansion from Theorem~\ref{ejbhw08}, we 
obtain the formula given in Corollary~\ref{ejbhw08v2}. Note that it is not valid for $\nu=\frac16n$, nor for $\nu=\frac56n$,
because at these points, independently of $n$, the argument of the logarithm equals one. But the definition of the asymptotic equivalence
does not imply the possibility to divide by a quantity, which is identically equal to zero.

Finally, setting $\nu=\frac12n$, $n$ is even, in the asymptotic expansion from Theorem~\ref{ejbhw08} and remarking that
\be\notag
\left.\frac{d^{2r-1}}{d\varphi^{2r-1}}\, \ctg\varphi \right|_{\varphi=\frac12\pi} \!=\,(-1)^r 
\frac{\,2^{2r-1} \left(2^{2r}-1 \right) B_{2r}\,}{r}\,, \qquad r\in\mathbbm{N}\,,
\ee
see e.g.~\cite[Corollary 1]{cvijovic_04}, as well as bearing in mind that $C_n=0$ for odd $n$, we arrive at Corollary~\ref{9784ryb}.
\end{proof}

\section{Bounds, inequalities, approximate equalities and asymptotic expansion}\label{2093u0j}
We already have bounds \eqref{98743cb38}, that we have obtained earlier; however, they are very rough and for many problems may be too inaccurate.
In Theorem~\ref{ejbhw082} we provide both upper and lower bounds, which are not only much sharper, but also asymptotically vanishing.
It may also be useful in some cases to have a suitable approximation for $C_n(\nu)$.
In Theorem~\ref{lk7d3mf4}, we give a simple approximation for $C_n(\nu)$, 
which can be useful for applications. Besides. it is well known that one and the same function may have asymptotic expansions involving 
different asymptotic sequences \cite[Chapt.~1]{erdelyi_01}, see also 
\cite{copson_01,dingle_01,evgrafov_03_eng,evgrafov_01_eng,olver_01}. Theorems~\ref{lk7d3mf5} and \ref{lk7d3mf5v2} give, in this sense,
alternative asymptotic expansions for $C_n(\nu)$,
which, in some situations, may be more desirable than the expansion obtained in Theorem~\ref{ejbhw08}, 
and which may be used as a very good approximation for $C_n(\nu)$ as well.
Since Theorems~\ref{lk7d3mf4}--\ref{lk7d3mf5v2} are closely related to each other, their proofs are given together.

\begin{theorem}[Bounds and inequalities for $\bm{C_n(\nu)}$]\label{ejbhw082}
For $n=2,3,4\ldots\,$ and $1<\nu<n$, the sum $C_n(\nu)$ may be bounded from below and from above in the following way: 
\be\notag
-\frac{\,2n\,}{\,\pi\,}\ln\!\left(2\sin\frac{\pi\nu}{n}\right)  +A(n,\nu) <
\sum_{l=1}^{n-1} \cos\frac{2\pi \nu l}{n}\cdot\csc\frac{\,\pi l\,}{n} < 
-\frac{\,2n\,}{\,\pi\,}\ln\!\left(2\sin\frac{\pi\nu}{n}\right)  +B(n,\nu),
\ee
where $\displaystyle A(n,\nu) \,\equiv\, - \,\frac{\pi }{\,12n\,}\csc^2\!\frac{\pi\nu}{n}\,$ and 
\be\notag
B(n,\nu) \,\equiv\,- \,\frac{\pi }{\,12n\,}\csc^2\!\frac{\pi\nu}{n} \, +
\, \frac{7\pi^3 }{\,1440\,n^3\,}\left\{1 \, +\,
2\cos^2\!\frac{\pi\nu}{n}\right\} \csc^4\!\frac{\pi\nu}{n}\,.
\ee
\end{theorem}

\begin{corollary}[Bounds and inequalities for $\bm{C_n}$]\label{4cvrc43fc}
For $n=2,4,6,\ldots\,$ the values of the alternating finite sum $\,\sum(-1)^{l+1}\csc\big(\pi l/n\big)$ 
always lie in the interval
\be\notag
\frac{\,2n\ln2\,}{\,\pi\,}
+\,\frac{\pi}{\,12\, n\,} \,-\,\frac{7\pi^3}{\,1440\,n^3\,} \,<
\sum_{l=1}^{n-1} (-1)^{l+1} \csc\frac{\,\pi l\,}{n}\,<\,\frac{\,2n\ln2\,}{\,\pi\,}
+\,\frac{\pi}{\,12\, n\,} 
\ee
For $n=3,5,7,\ldots\,$ this finite sum vanishes.
\end{corollary}

\begin{proof}
Differentiating $m-1$ times \cite[Vol.~I, Sec.~1.9, Eq.~(10)]{bateman_01} with respect to $z$, we get 
\be\notag
\Psi_{m}(x)\,=\,(-1)^{m+1} \, m! \sum_{l=0}^\infty \frac{1}{\,(x+l)^{m+1}} \,=\,(-1)^{m+1} \, m! \,\zeta(m+1,x)\,,
\qquad m\in\mathbbm{N}\,.
\ee
Hence $\,\Psi_{2r-1}(x)>0\,$ for $x>0$. Therefore $\,\Psi_{2r-1}(x)+\Psi_{2r-1}(1-x)>0\,$ for $0<x<1$. 
Considering now the definition of $\mathscr{H}_{2r-1}(\nu/n) $ given in Theorem~\ref{ejbhw08}, we see at once that
\be\notag
\mathscr{H}_{2r-1}\left(\frac{\nu}{n}\right) < 0\,,
\ee
whence
\be\notag
\sgn\!\left[ \frac{\,\big(1-2^{1-2r}\big)\pi^{2r-1} B_{2r}\,}{\,(2r)!\, n^{2r-1}\,} 
\mathscr{H}_{2r-1}\left(\frac{\nu}{n}\right)  \right] =\,(-1)^r\,.
\ee
Thus, the series in the right--hand side of \eqref{8943ycn9438} possesses the usual property 
concerning the magnitude and sign of the remainder. 
Setting $N=2$ and $N=3$ into \eqref{8943ycn9438} and accounting for the sign yields both inequalities 
stated in Theorem~\ref{ejbhw082}.
By a similar line of reasoning, we deduce the result announced in Corollary~\ref{4cvrc43fc}.
\end{proof}

\begin{theorem}[A simple approximation for $\bm{C_n(\nu)}$]\label{lk7d3mf4}
A simple and relatively good approximation for the sum $C_n(\nu)$ is given by the following expression
\be\notag
\sum_{l=1}^{n-1} \cos\frac{2\pi \nu l}{n}\cdot\csc\frac{\,\pi l\,}{n} 
\,\approx  -\frac{\,2n\,}{\,\pi\,}\ln\!\left(2\sin\frac{\pi\nu}{n}\right)
 \, - \,\frac{\pi }{\,12n\,}\csc^2\!\frac{\pi\nu}{n} \,+\, \frac{7n}{\,480\pi \nu^4\,}\,.
\ee
This approximation is quite accurate, the right--hand side is asymptotically equivalent to $C_n(\nu)$ 
at large $n$, but one should bear in mind that 
the approximation error does not tend to zero as $n\to\infty$.
\end{theorem}

\begin{theorem}[An asymptotic expansion and accurate approximation for $\bm{C_n(\nu)}$]\label{lk7d3mf5}
The sum $C_n(\nu)$, $1<\nu<n$, admits the following asymptotic expansion
\be\notag
\begin{array}{ll}
\displaystyle 
\sum_{l=1}^{n-1} \cos\frac{2\pi \nu l}{n}\cdot\csc\frac{\,\pi l\,}{n} 
&\displaystyle = \, -\frac{\,2n\,}{\,\pi\,}\ln\!\left(2\sin\frac{\pi\nu}{n}\right)
 \, - \,\frac{\pi }{\,12n\,}\csc^2\!\frac{\pi\nu}{n} \,+\, n f(\nu) \, + \\[6mm]
&\displaystyle\qquad
+ \frac{7\pi^3}{\,21\,600 \,n^3\,}  +\, o\big(n^{-3}\big)\,,\qquad\quad n\to\infty\,,
\end{array}
\ee
where 
\be\notag
f\left(\nu\right) =-\frac{1}{\,\pi\,}\left\{4H_{2\nu}-2H_{\nu}-2\ln\nu-4\ln2-2\gamma-\frac{1}{12\nu^2} \right\}\,.
\ee
It may also be used as a very accurate approximation for $C_n(\nu)$, whose error 
rapidly tends to zero as $n\to\infty$.
\end{theorem}

\begin{theorem}[A cosecant--free asymptotic expansion of $\bm{C_n(\nu)}$]\label{lk7d3mf5v2}
The sum $C_n(\nu)$, $1<\nu<n$, admits the following cosecant--free asymptotic expansion
\be\notag
\begin{array}{ll}
\displaystyle 
\sum_{l=1}^{n-1} \cos\frac{2\pi \nu l}{n}\cdot\csc\frac{\,\pi l\,}{n} 
&\displaystyle = \, -\frac{\,2n\,}{\,\pi\,}\ln\!\left(2\sin\frac{\pi\nu}{n}\right)
 \,+\, n g(\nu) \, - \frac{\pi}{\,36 \,n\,} - \\[6mm]
&\displaystyle\qquad
- \frac{\,\left(120\nu^2-7\right)\pi^3\,}{\,21\,600 \,n^3\,}  +\, o\big(n^{-3}\big)\,,\qquad\quad n\to\infty\,,
\end{array}
\ee
where 
\be\notag
g\left(\nu\right) =-\frac{1}{\,\pi\,}\Big\{4H_{2\nu}-2H_{\nu}-2\ln\nu-4\ln2-2\gamma \Big\}\,.
\ee
It may also be used as an approximation for $C_n(\nu)$, whose error 
tends to zero as $n\to\infty$.
\end{theorem}

\begin{figure}[!t]   
\centering
\includegraphics[width=0.8\textwidth]{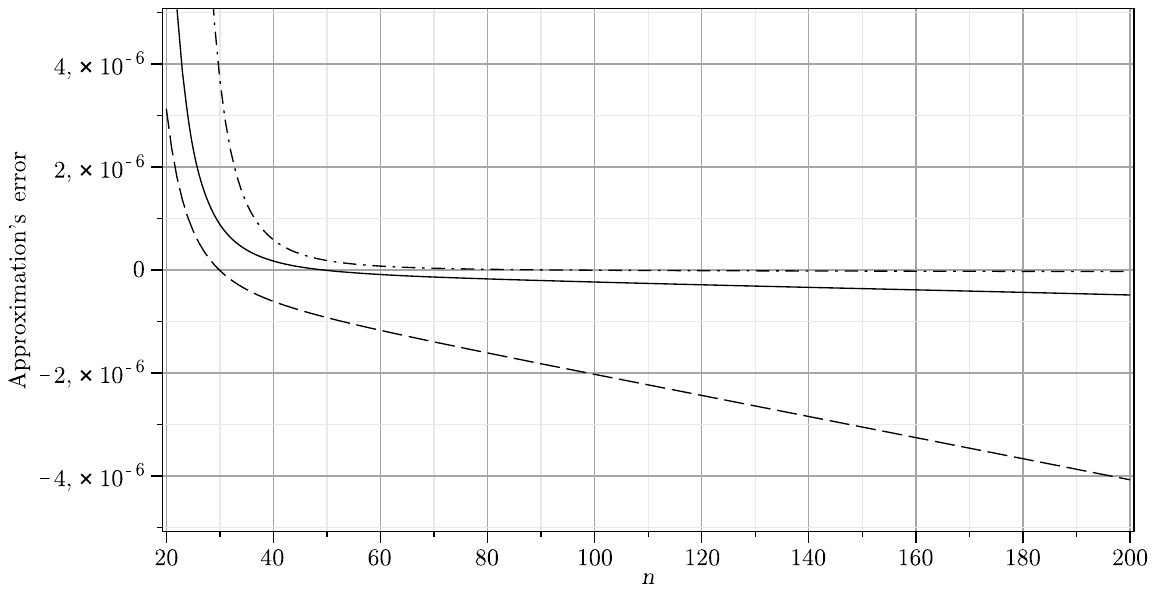}
\vspace{-0.5em}
\caption{The difference between $C_n(\nu)$ and its approximation, provided by
Theorem~\ref{lk7d3mf4}, as a function of $n$ for $\nu=7$ (dashed line), $\nu=10$ (solid line)
and $\nu=16$ (dash--dotted line). Note that the greater the argument $\nu$, the better the approximation
[roughly the approximation error is proportional to $\nu^{-6}$]. 
One should also bear in mind that $C_{100}(7)\approx53$,
$C_{100}(10)\approx31$ and $C_{100}(16)\approx2$, so that in all these cases the approximation is rather
accurate (and it remains such at least for moderate values of $n$, e.g.~for $n=10\,000$ and $\nu=7$
the approximation error is about $-2\times10^{-4}$, while $C_{10\,000}(7)\approx34\,541$).
}
\label{937fybfe4}
\end{figure}

\begin{corollary}\label{lpo3sai}
Let $n$ be a multiple of 6. If $\nu=\frac16n$ or $\nu=\frac56n$, then as $n\to\infty$ the sum $C_n(\nu)$ tends to zero:
\be\notag
\lim_{n\to\infty}\sum_{l=1}^{n-1} \cos\frac{\pi l}{3}\cdot\csc\frac{\,\pi l\,}{n} \,=\,0\,. 
\ee
\end{corollary}

\begin{proof}
From the inequalities established in the previous theorem, it follows that
\be\label{68v76ugi}
\begin{array}{ll}
\displaystyle 
C_n(\nu)   &\displaystyle = \,  -\frac{\,2n\,}{\,\pi\,}\ln\!\left(2\sin\frac{\pi\nu}{n}\right)
 \, - \,\frac{\pi }{\,12n\,}\csc^2\!\frac{\pi\nu}{n} \,+\, \\[6mm]
&\displaystyle\qquad \qquad \qquad 
\,+\, \frac{7\lambda\pi^3 }{\,1440\,n^3\,}\left\{1 \, +\,
2\cos^2\!\frac{\pi\nu}{n}\right\} \csc^4\!\frac{\pi\nu}{n} \, , 
\end{array}
\ee
where $0<\lambda<1$. Expanding the last term into the power series in $n$, $n\to\infty$, we have 
\be\label{8uyb098yb}
\frac{7\pi^3 }{\,1440\,n^3\,}\left\{1 \, +\,
2\cos^2\!\frac{\pi\nu}{n}\right\} \csc^4\!\frac{\pi\nu}{n} \,=\, \frac{7n }{\,480\pi \nu^4\,}
+\frac{7\pi^3}{\,21\,600 \, n^3\,}+O\big(n^{-5}\big)\,,
\ee
where the expression in the left--hand side is clearly positive.
Substituting this result into \eqref{68v76ugi}, we obtain
\be 
C_n(\nu)
\,\approx  -\frac{\,2n\,}{\,\pi\,}\ln\!\left(2\sin\frac{\pi\nu}{n}\right)
 \, - \,\frac{\pi }{\,12n\,}\csc^2\!\frac{\pi\nu}{n} \,+\, n f(\nu)\, , 
\ee
where $f\left(\nu\right)$ is a bounded function 
\be\notag
0< f\left(\nu\right) < \frac{7}{\,480\pi \nu^4\,}\,.
\ee
Proceeding analogously with the term corresponding to $r=3$ in \eqref{8943ycn9438},
whose asymptotics
\be
\sim \, -\frac{31}{\,4032\pi \nu^6\,}\,,\qquad n\to\infty\,,
\ee
we see that at sufficiently large $n$ and fixed $\nu$ the function $f(\nu)$ obeys these inequalities 
\be
\frac{7}{\,480\pi \nu^4\,}-\frac{31}{\,4032\pi \nu^6\,}  < f\left(\nu\right) < \frac{7}{\,480\pi \nu^4\,}\,.
\ee
But numerically, such a correction from below (due to the term with $r=3$) is almost negligible (e.g.~for $\nu\geqslant8$ it is lesser than $1\%$).
Hence, $f(\nu)$ is practically equal to its upper bound, whence we get the approximation stated in Theorem~\ref{lk7d3mf4}.

\begin{figure}[!t]   
\centering
\includegraphics[width=0.8\textwidth]{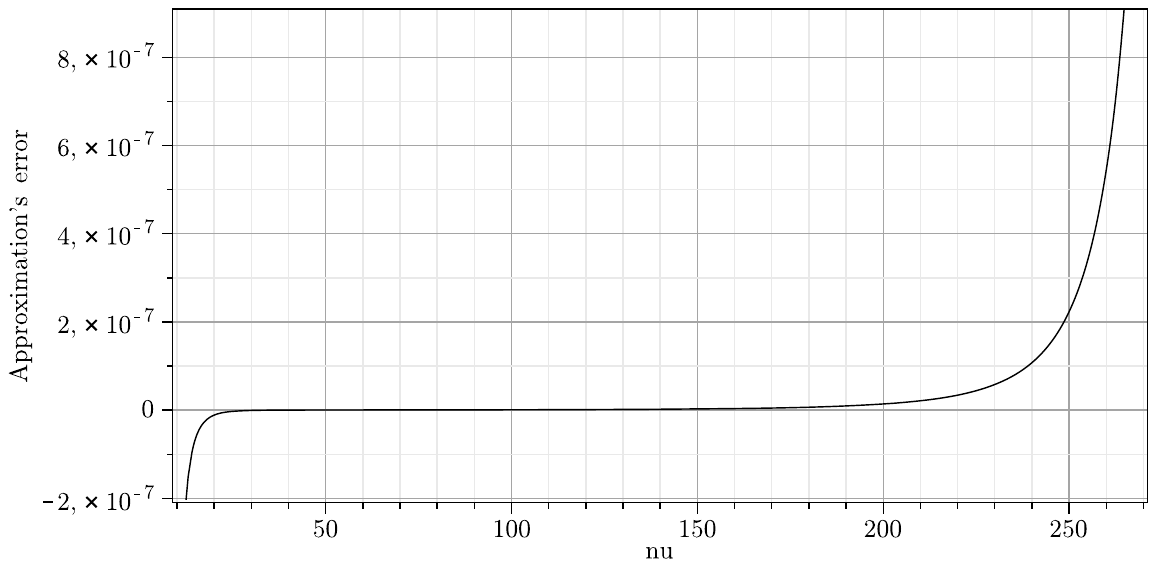}
\vspace{-0.5em}
\caption{The approximation error for $C_{300}(\nu)$ as a function of $\nu$, where $\nu\in[10,270]$
(approximation given by Theorem~\ref{lk7d3mf4}).
For the sake of comparison: 
$C_{300}(10)\approx+299$,
$C_{300}(20)=C_{300}(280)\approx+168$,
$C_{300}(50)=C_{300}(250)\approx-3\times10^{-3}$,
$C_{300}(100)=C_{300}(200)\approx-105$,
$C_{300}(150)\approx-132$ (see Fig.~\ref{g6hytbhfw2} for the graph of $C_{300}(\nu)$).
}
\label{937fybfe5}
\end{figure}

Obviously, we can also 
take into account the contribution of higher terms in the asymptotic expansion from Theorem~\ref{ejbhw08}. 
Remarking that for a sufficiently large $n$ 
\be\label{83490cn834}
\left.\frac{d^{2r-1}}{d\varphi^{2r-1}}\, \ctg\varphi \right|_{\varphi=\frac{\pi\nu}{n}} \!=
\,-\frac{n^{2r}(2r-1)!}{(\pi\nu)^{2r}}+K_r+O\big(n^{-2}\big)\,, \qquad r\in\mathbbm{N}\,,
\ee
where $K_r$ is a constant term, depending neither on $n$ nor on $\nu$, we obtain
\be\label{98432fny89}
f(\nu)\,=\,-\frac{1}{\,\pi\,}\!\sum_{r=2}^{N-1} \frac{\,\big(1-2^{1-2r}\big) B_{2r}\,}{\,r\, \nu^{2r}\,} +\ldots
\ee
Recalling the Stirling formula for the harmonic numbers
\be\notag
H_n\,= \ln n \,+\, \gamma \,+\, \frac{1}{\,2n\,} \,-\, \frac{1}{\,2\,}\!\sum_{r=1}^{N-1} \frac{B_{2r}}{\,r\,n^{2r}\,}
\,-\, \frac{\varepsilon \, B_{2N}}{\,2N n^{2N}\,}\,, \qquad\qquad  0<\varepsilon<1\,,
\ee
we see that \eqref{98432fny89} is the asymptotic expansion of the difference of two harmonic numbers
with some additional terms, \emph{viz.}
\be\label{894cb937832x}
\begin{array}{lll}
\displaystyle 
f(\nu)\,&\displaystyle\: =\,-\frac{1}{\,\pi\,}\left\{\sum_{r=2}^{N-1} \frac{\, B_{2r}\,}{\,r\, \nu^{2r}\,}
-2\!\sum_{r=2}^{N-1} \frac{\, B_{2r}\,}{\,r \left(2\nu\right)^{2r}\,} \right\} +\ldots \\[7mm]
&\displaystyle =\,-\frac{1}{\,\pi\,}\left\{4H_{2\nu}-2H_{\nu}-2\ln\nu-4\ln2-2\gamma-\frac{1}{12\nu^2} \right\}
\approx \, \frac{7}{\,480\pi \nu^4\,}\,.
\end{array}
\ee
Note that since $\frac{7}{\,480\pi \nu^4\,}$
is just an approximation for $f(\nu)$, the overall approximation error, given by Theorem~\ref{lk7d3mf4}, linearly grows 
with $n$, but with a very small slope, which roughly is inversely proportional to $\nu^{6}$.
In contrast, the approximation error of the alternative asymptotic expansion given by Theorem~\ref{lk7d3mf5}
does tend to $0$ as $n\to\infty$. In fact, the latter asymptotic expansion 
does not contain terms $O(1)$, nor $O(n^{-1})$ nor even $O(n^{-2})$. As to the order $n^{-3}$,
this term is obtained from the constant term in \eqref{83490cn834} $K_2=-2/15$, and from the corresponding contribution 
of the sum in the right--hand side of \eqref{8943ycn9438}. Note that this term does not depend on $\nu$ at all,
and thus, may be regarded as a small \emph{bias} if we study $C_n(\nu)$ at large fixed $n$. 
Furthermore, it can be reasonably expected that in virtue of \eqref{83490cn834},
the remaining terms in the asymptotics of $C_n(\nu)$ should be $O\!\left(n^{-5}\right)$,
the statement which can be readily verified empirically. 

Finally, in some cases, the presence of the square of the cosecant 
in the asymptotic expansion of $C_n(\nu)$ may be undesirable. In such a situation, we may get rid of it
by expanding the cosecant term into power series
\be\notag
\csc^2\!\frac{\pi\nu}{n} \,=\, \frac{n^2 }{\,\pi^2 \nu^2\,} +\frac13
+\frac{\pi^2\nu^2}{\,15\, n^2\,}+\frac{\,2\pi^4\nu^4\,}{\,189\, n^4\,}+O\big(n^{-6}\big)\,,\qquad n\to\infty.
\ee
Inserting this expansion into Theorem~\ref{lk7d3mf5} yields the cosecant--free expansion 
stated in Theorem~\ref{lk7d3mf5v2}. 
\end{proof}

\begin{figure}[!t]   
\centering
\includegraphics[width=0.8\textwidth]{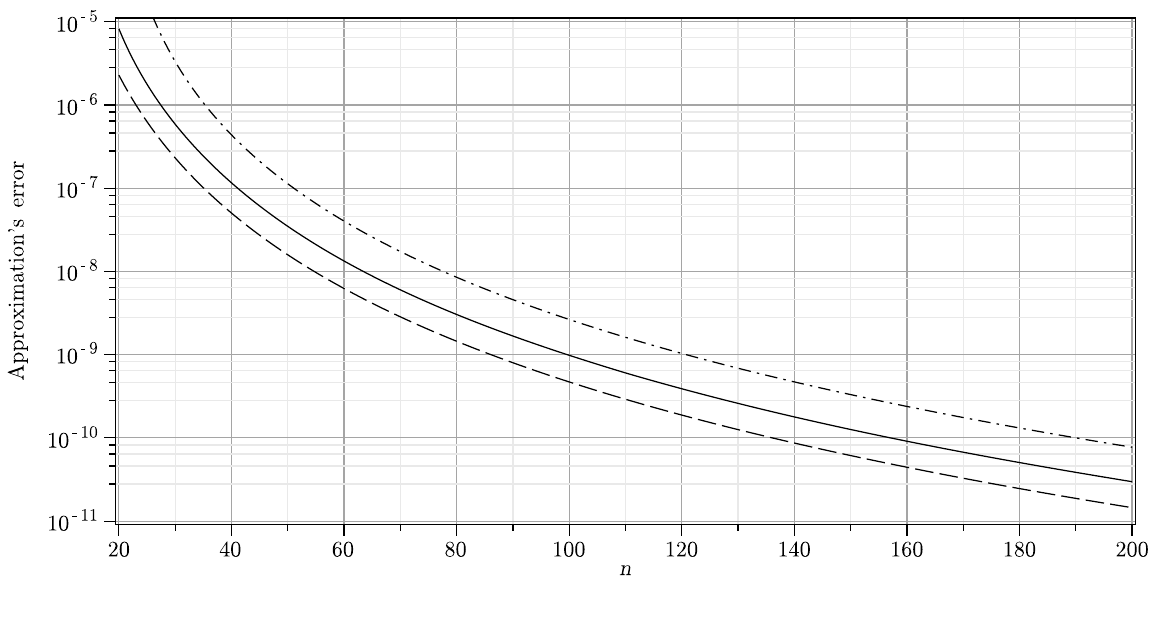}
\vspace{-1.2em}
\caption{The difference between $C_n(\nu)$ and its approximation, provided by the asymptotics from
Theorem~\ref{lk7d3mf5}, as a function of $n$ for $\nu=7$ (dashed line), $\nu=10$ (solid line)
and $\nu=16$ (dash--dotted line). It is clearly visible that the approximation error quickly tends to zero and that it is very small
(compare this graph to Fig.~\ref{937fybfe4}).
}
\label{937fybfe6}
\end{figure}

\begin{figure}[!t]   
\centering
\includegraphics[width=0.8\textwidth]{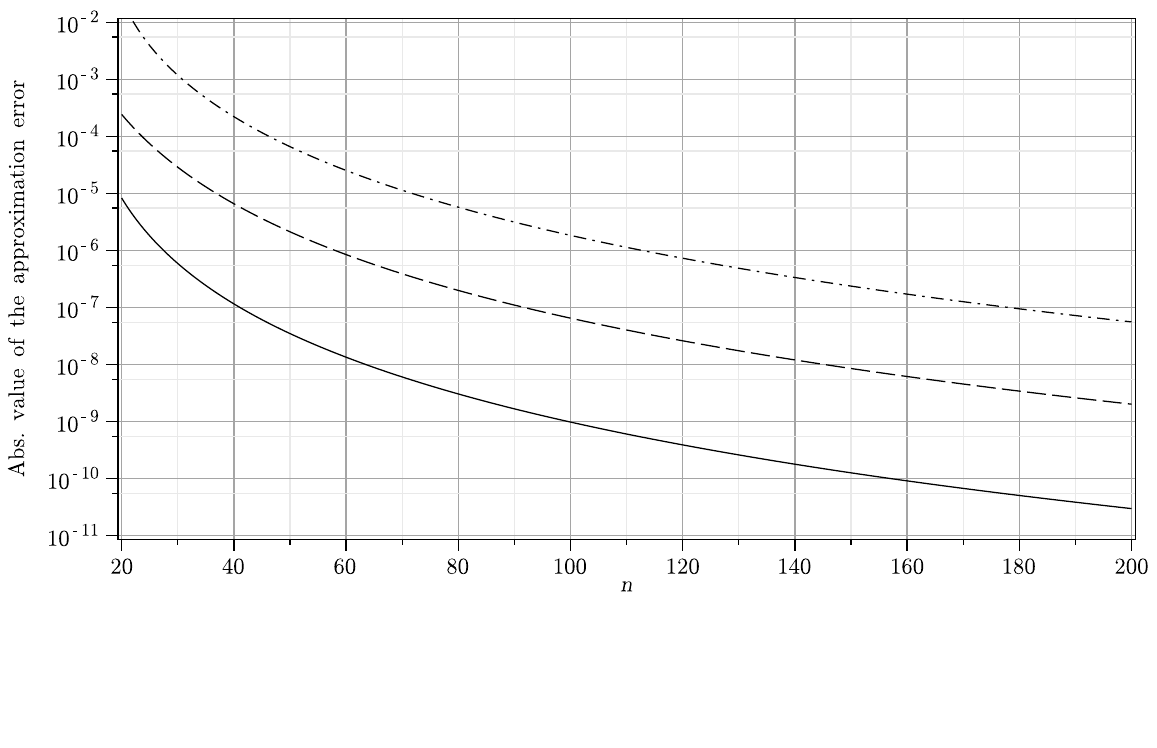}
\vspace{-4em}
\caption{The absolute difference between $C_n(\nu)$ and its approximation, offered by the 
cosecant--free asymptotics from Theorem~\ref{lk7d3mf5v2}, as a function of $n$ for $\nu=7$ (dashed line), 
$\nu=10$ (solid line) and $\nu=16$ (dash--dotted line).}
\label{937fybfe7}
\end{figure}

We conclude this Section with several graphs, showing how well the approximate formula, as well as the alternative asymptotics, represent
the sum $C_n(\nu)$. The difference between $C_n(\nu)$ and its approximation, provided by
Theorem~\ref{lk7d3mf4}, is shown in Figs.~\ref{937fybfe4} and~\ref{937fybfe5}.
Fig.~\ref{937fybfe4} displays the approximation error as a function of $n$ for three different values of argument $\nu$.
We see that the approximation error is very small, so that it can be used for many versatile purposes and applications.
At the same time, as expected, we observe that it does not tend to zero as $n$ increases; however, even in the worst case, that of $C_{n}(7)$,
the approximation error is still very small even for large values of $n$ (note that $\nu^{6}=7^6=117\,649$).
Fig.~\ref{937fybfe5} shows the approximation error as a function of $\nu$ at fixed $n$.
Fig.~\ref{937fybfe6} displays the error between $C_n(\nu)$ and its alternative asymptotics, offered by Theorem~\ref{lk7d3mf5}
in the same conditions as in Fig.~\ref{937fybfe4}, i.e. for the same three arguments $\nu$ and in the same interval of $n$. 
We see that the error is extremely small and quickly tends to zero as $n$ increases. In other words,
the asymptotics from Theorem~\ref{lk7d3mf5} may also be used 
as a very accurate approximation for $C_n(\nu)$, but the other side of the coin is the complexity of calculations, 
which may considerably grow in size or even become prohibitive. Lastly, as to the less complex 
approximation for $C_n(\nu)$, provided by the cosecant--free asymptotics from Theorem~\ref{lk7d3mf5v2},
we see, Fig.~\ref{937fybfe7}, that it is less accurate than that provided by Theorem~\ref{lk7d3mf5},
but the error still tends, though not very quickly, to zero as $n\to\infty$; at the same time, it remains more 
accurate than the approximation provided by Theorem~\ref{lk7d3mf4}.

Finally, as to the result given in Corollary~\ref{lpo3sai}, it is a simple consequence of Theorem~\ref{lk7d3mf5}. 
Setting $\nu=\frac16n$ or $\nu=\frac56n$, we see that the leading term in the asymptotics, 
provided in Theorem~\ref{lk7d3mf5}, identically vanishes. The remaining terms are $o(1)$, whence
we obtain the stated result. These two cases are the only ones when the sum $C_n(\nu)$ 
converges as $n\to\infty$ (in our case, trivially to zero, since there are no constant terms 
in the asymptotic expansion of $C_n(\nu)$).

\small
\bibliographystyle{crelle}

\end{document}